\documentclass[10pt,a4paper]{article}
\usepackage[utf8]{inputenc}
\usepackage{amsmath}
\usepackage{amsthm}
\usepackage{amsfonts}
\usepackage{amssymb}
\usepackage{graphicx}
\usepackage{multirow}
\usepackage{epstopdf}
\usepackage{epsfig}
\usepackage{algorithm,algpseudocode}
\usepackage[legalpaper, margin=1.5in]{geometry}

\usepackage{anyfontsize}

\def\PR{ {\scalebox{.5}{\textbf{PR}}}  }

\def\v{{\mathbf v}} \def\y{{\mathbf y}} \def\x{{\mathbf x}}

  \def\e{{\mathbf e}}

\def\s{{\mathbf s}}  \def\g{{\mathbf g}}

 \def\RR{{\mathbb R}} \def\A{{\mathcal A}}

\def\V{{\mathcal V}} \def\P{{\mathcal P}} 

\def\bar{\overline} \def\tilde{\widetilde}

{\color{blue} }%
{\color{black} }

\title{Extrapolation Methods for fixed-point Multilinear {P}age{R}ank computations}

\author{Stefano Cipolla\footnotemark[1] , Michela Redivo-Zaglia\footnotemark[1] , Francesco Tudisco\footnotemark[2]}

\newtheorem{definition}{Definition}
\newtheorem{proposition}{Proposition}
\newtheorem{theorem}{Theorem}
\newtheorem{lemma}{Lemma}
\begin{document}
	\maketitle
	\renewcommand{\thefootnote}{\fnsymbol{footnote}}
	\footnotetext[1]{Department of Mathematics ``Tullio Levi-Civita'', University of Padua, Padova, Italy (\texttt{stefano.cipolla@unipd.it}, \texttt{michela.redivozaglia@unipd.it}) }
	\footnotetext[2]{School of Mathematics and Statistics, University of Edinburgh, EH9 3FD  Edinburgh, UK (\texttt{tudisco.francesco@gmail.com})}

	\begin{abstract}{Nonnegative tensors arise very naturally in many applications that involve large and complex data flows. Due to the relatively small requirement in terms of memory storage and number of operations per step, the (shifted) higher-order power method is one of the most commonly used technique for the computation of positive $Z$-eigenvectors of this type of tensors.  However, unlike the matrix case, the method may fail to converge even for irreducible tensors. Moreover, when it converges, its convergence rate can be very slow. These two drawbacks often make the computation of the eigenvectors   demanding or unfeasible for large problems. In this work we consider a particular class of nonnegative tensors associated to the multilinear {P}age{R}ank modification of higher-order Markov chains. Based on the simplified topological $\varepsilon$-algorithm in its restarted form, we introduce an extrapolation-based acceleration of power method type algorithms, namely the shifted fixed-point method and the inner-outer method. The accelerated methods show remarkably better performance, with faster convergence rates and reduced overall computational time. Extensive numerical experiments on synthetic and real-world datasets demonstrate the advantages of the introduced extrapolation techniques.}

\end{abstract}


\section{Introduction}

A higher-order Markov chain with memory of length $m-1$ is a stochastic process $\{S_t\}_{t=0}^\infty$ over a finite state space $V=\{1,\dots,n\}$ where the conditional probability distribution of the next state in the process depends only on the last $m-1$ states.  As for standard chains with memory of length 1, a higher-order Markov chain with a memory of length $m-1$ can be represented by the $m$-order transition tensor 
\begin{equation*}
p_{{i_1},{i_2 \dots, i_{m}}}:=\mathbb{P}(S_{t}={i_1}|S_{t-1}={i_2},\cdots,S_{t-(m-1)}={i_{m}}),\qquad \forall \;t \, .
\end{equation*}

 Note that by definition we obtain  $p_{{i_1},{i_2 \dots, i_{m}}} \geq 0$ and	$\sum_{{i_1}=1}^np_{{i_1},{i_2 \dots, i_{m}}}=1$ for all $({i_{2}, \dots, i_{m}})$. The stationary distribution $s_{i_1,\dots,i_{m-1}}$ of the chain is thus given by
\begin{equation}\label{eq:ho-stationary-dist}
\sum_{i_m=1}^n p_{i_1,\dots,i_m}s_{i_2,\dots,i_{m}} = s_{i_1,\dots,i_{m-1}} \, .
\end{equation}
As on the one hand, a longer memory has the advantage of allowing  more accurate models and thus may offer additional predictive value, on the other hand, computing a stationary distribution  for a Markov chain with memory length $m-1$ requires the computation of $n^{m-1}$ parameters and thus the memory requirement and the computational cost to treat  the 
 stationary distribution \eqref{eq:ho-stationary-dist} become extremely demanding already for  moderate values of $m$ and $n$ \cite{wu2017markov}. %

In this work we are particularly interested in {P}age{R}ank-type chains.
In the standard single-order setting, given a random walk on a directed graph {with $n$ nodes, the {P}age{R}ank modification builds a new Markov chain that  has always a unique stationary distribution. {Since its introduction \cite{page1999pagerank}, the analysis of the PageRank random walk and its stationary distribution has given rise to a wide literature, {with applications going further beyond the initial treatment of the web hyperlinks network \cite{gleichReview}.} In particular, the very large size of typical problems where the PageRank is applied has led to the development of several algorithms for its numerical treatment, ranging from specialized versions of the power, Jacobi and Richardson methods \cite{berkhin2005survey,cipolla2017euler,tudisco2011preconditioning,del2005fast,kamvar2004adaptive}, to algebraic and Krylov-type methods \cite{golub2006arnoldi,serra2005jordan}. In particular, the use of extrapolation strategies has shown to be very effective in this context \cite{brezinski2006pagerank,brezinski2008rational,brezinski2005extrapolation,kamvar2003extrapolation}.}%

Recently, the PageRank idea has been extended to higher-order Markov chains \cite{gleich2015}.
However, due to the exponential complexity required by the higher-order setting, the authors further introduce a computationally tractable approximation of the true higher-order {P}age{R}ank distribution, called multilinear {P}age{R}ank.
From the algebraic point of view, the multilinear {P}age{R}ank  seeks a rank-one stationary distribution of the higher-order chain, i.e., the joint probability $s_{i_1,\dots,i_{m-1}}$ solution of \eqref{eq:ho-stationary-dist} is replaced by a rank-one tensor $s_{i_1}s_{i_2}\cdots s_{i_{m-1}}$. This formulation dramatically reduces the dimension of the solution set from $n^{m-1}$ to $n$. Moreover, it has been proved in \cite{li2014limiting} that, under the assumption $s_{i_1,\dots,i_{m-1}}= s_{i_1}s_{i_2}\cdots s_{i_{m-1}}$, the solution of \eqref{eq:ho-stationary-dist} coincides with a $Z$-eigenvector of the transition tensor, given by
\begin{equation}\label{eq:Z-eig}
\sum_{i_{2},\dots,i_{m}}p_{i_1,\dots,i_m}s_{i_2}\cdots s_{i_m} = s_{i_1},\qquad \forall i_1 = 1,\dots,n\, .
\end{equation}

Motivated by the success  that extrapolation methods have had in the computation of the single-order PageRank, in this work we propose and investigate the use of extrapolation methods, based on Shanks transformations \cite{brezinski2018shanks,shanks1955non}, for the particular setting of tensor $Z$-eigenpairs problems arising from the analysis of the multilinear {P}age{R}ank.
More precisely, after a short survey on results about  existence and uniqueness of the solution and on the state-of-the-art fixed-point computational methods for the multilinear {P}age{R}ank vector, we will show how its computation can be considerably sped-up using extrapolation techniques. In particular, we will show that the sequences generated by the two fixed-point-type techniques \textit{Shifted Higher-Order Power Method}  \cite{gleich2015,kolda2011shifted} and \textit{Inner-Outer Method} \cite{gleich2015}, are strongly accelerated 	using the \textit{The Simplified Topological $\varepsilon$-Algorithm} (STEA) \cite{brezinski2014simplified,brezinski2017simplified} in the restarted form.
Alongside fixed-point methods, techniques based on the Newton's method have been proposed for instance in \cite{gleich2015,meini2017perron}. 
The  results presented in this work  show that the use of the extrapolation framework based on the Simplified Topological $\varepsilon$-Algorithm  allows us to improve the efficiency and the robustness of fixed-point iterations without resorting on the Newton's framework which requires, for large scale problems, a severe computational cost.

The use of simplified topological extrapolation techniques for this type of problem is, to our knowledge, a useful and effective novelty. For example, it gives a positive answer to the question (quoting from \cite{kolda2011shifted}) 
{\it Can the convergence rate of the current SS-HOPM  method be accelerated?} The introduction of the extrapolation framework and the acceleration techniques here considered not only produce a  relevant speed-up but  also enhance the robustness of the method without modifying the overall computational cost of the underlying iterative procedure (see Section \ref{sec:topological_epsilon_algorithm}).
Moreover, as the multilinear {P}age{R}ank is an instance of the many eigenvector problems that correspond to nonnegative  tensors \cite{gautier2018unifying}, the results proposed here can be transferred  to a variety of data analysis applications  where nonnegative tensors  and their spectra  play an important role. Examples, involving nonnegative tensors in general and the multilinear {P}age{R}ank in particular, include ranking of nodes in multilayer networks \cite{ng2011multirank,tudisco2017node,arrigo2018multi}, data clustering  \cite{benson2015tensor,wu2017markov}, hypergraph matching \cite{chertok2010efficient,nguyen2017efficient}, computer vision \cite{shashua2005non} and image reconstruction \cite{soltani2016tensor}. Although much has been done in  recent years, much is still unknown for tensor eigenpairs, including the development of fast algorithms for their computation. To this end,  higher-order and nonlinear versions of the classic power method  have been proposed  to address the computation of different types of tensors eigenpairs (see e.g. \cite{friedland2013perron,gautier2018perron,gautier2018unifying,kofidis2002best,kolda2011shifted,ng2009finding}) and, due to the large size of typical problems, enhancing the efficiency and robustness of these methods is of utmost importance.

The reminder of this paper is organized as follows: in Section \ref{sec:theoretical_background} we introduce and review all the necessary theory concerning the multilinear PageRank problem and related computational { techniques and we   analyze the existence and uniqueness issue of the multilinear Pagerank problem}.  In Section \ref{sec:extr_framework}, after introducing and reviewing the topological extrapolation framework, we describe our proposed simplified topological extrapolation method with restart and detail how we apply it to  the specific multilinear PageRank case;
finally,  in Section \ref{sec:numerical_results} we present extensive numerical results on synthetic and real-world datasets that showcase the effectiveness and the advantages of our approach.

\section{Multilinear {P}age{R}ank} \label{sec:theoretical_background}
\subsection{Notation}	
We say that $\mathcal A$ is a real cubical tensor of order $m$ and dimension $n$ if $\mathcal A$ is a multi-dimensional array with real entries  such that
$$
\mathcal A = (a_{i_1,\dots,i_m}) \quad \text{with} \quad  i_1,\dots,i_m \in \{1,\dots,n\} \, .
$$
Given a vector $\s\in \RR^n$ we denote by $\mathcal A\s^{m-1}\in \RR^n$ the vector with entries
$$
(\mathcal A \s^{m-1})_{i_1} = \sum_{i_2,\dots, i_m=1}^n a_{i_1,i_2,\dots, i_m}s_{i_2}\dots s_{i_m}, \qquad i_1=1,\dots, n
$$
A vector $\s = (s_i)\in \RR^n$ is said to be stochastic or, equivalently, a probability distribution, if $s_i\geq 0$ for all $i=1,\dots, n$ and $\sum_i s_i = 1$. Similarly, a real cubical tensor $\mathcal A$ is said to be stochastic if its entries are nonnegative numbers and the entries on the first mode of $\mathcal A$  sum up to one, namely

\begin{equation*}
a_{i_1,\dots, i_m} \geq 0, \quad \forall i_1,\dots,i_m\in \{1,\dots,n\}\qquad \text{and} \qquad \sum_{i_1=1}^n a_{i_1,\dots, i_m}=1 \, .
\end{equation*} 	
Unlike the matrix case, a number of different types of tensor eigenvectors can be defined, see for instance \cite{gautier2018unifying} and the references therein. 

Here we are concerned with $Z$-eigenvectors: a number $\lambda \in \mathbb C$ is an $E$-eigenvalue  of $\A$ with $E$-eigenvector $\s\in \mathbb C^n$ if it holds
\begin{equation}\label{eq:Z-eig}
\A \s^{m-1} = \lambda \, \s \qquad \text{and} \qquad \|\s\|_1=\sum_{i=1}^n|s_i| = 1 \,.
\end{equation}
A real $E$-eigenvector $\s\in \mathbb R^n$ is called $Z$-eigenvector and the corresponding $\lambda$ a $Z$-eigenvalue. We point out that, although the normalization constraint on $\s$ in the definition \eqref{eq:Z-eig} is more commonly given in terms of the Euclidean norm $\|\cdot \|_2$, we require here the $1$-norm for the sake of convenience. In fact, this is the choice that is more natural when dealing with stochastic tensors, as in that case it easy to see that  the only $Z$-eigenvalue $\lambda$ of a nonnegative $Z$-eigenvector is $\lambda = 1$. For reference, we also point out that $\s$, solution of \eqref{eq:Z-eig}, is sometimes called $Z_1$-eigenvector to underline the choice of the one norm (see e.g. \cite{chang2013uniqueness}). We avoid this additional notation, for the sake of simplicity.

Finally, as for nonnegative matrices, the concept of (nonnegative) irreducible tensor is important in order to show existence and uniqueness of nonnegative tensor eigenvectors. Again, unlike the matrix case, the concept of irreducible tensor is not ``universal'' but strongly depends on the structure of $\A$ and on the eigenvector problem one is interested in (see \cite{gautier2018unifying} for a thorough discussion on the topic). In this work we focus on the following notion of irreducible cubical tensors, originally proposed in~\cite{chang2008perron}
\begin{definition}
	A $n$-dimensional order-$m$ nonnegative cubical tensor $\A$ is called \emph{reducible} if there exists a set of indices $I\subseteq V =\{1,\dots,n\}$, $I\neq \emptyset$, $I\neq V$, such that
	$$
	a_{i_1,i_2, \dots, i_m} = 0, \qquad \forall i_1 \in I, \qquad \forall i_2, \dots, i_m \notin I\, .
	$$
	The tensor $\A$ is irreducible if it  is not reducible.
\end{definition}

\subsection{Background}
Let $\Omega\subseteq \RR^n$ be the set of all stochastic vectors
\begin{equation*}
\Omega:=\Big\{ \s \in \mathbb{R}^n\;:\; \sum_{i=1}^n s_i=1, \, {s}_i \geq 0,  i=1,\dots,n\Big\},
\end{equation*}
and let $\Omega_+\subseteq \Omega$  be the set of entry-wise positive stochastic vectors.

The multilinear {P}age{R}ank model proposed in \cite{gleich2015} seeks a fixed-point solution  in $\Omega$ of the following nonlinear map:
\begin{equation} \label{prob:mpr}
f(\s)=\alpha\P\s ^{m-1}+(1-\alpha)\v,
\end{equation}
being $\P = (p_{i_1,\dots,i_m})$ a stochastic tensor ($\sum_{i_1}p_{i_1,\dots,i_m}=1$) and $\v \in \Omega_+$  the so-called ``teleportation vector''.
It is worth point out that a solution of  \eqref{prob:mpr} is a stationary distribution of a stochastic process called ``The Spacey Random Surfer'' \cite{benson2017spacey} which is an interesting vertex-reinforced Markov process that uses a combination of {an} {aggregated history} and the current state.

The map \eqref{prob:mpr} is reminiscent of the standard single-order {P}age{R}ank model. In that case, given the transition matrix  $P=(p_{ij})$, $\sum_i p_{ij} = 1$, of a random walk on a graph with $n$ nodes $V=\{1,\dots,n\}$ and given $\v=(v_1,\dots,v_n) \in \Omega_+$, one seeks to compute a positive eigenvector of the {P}age{R}ank transition matrix
\begin{equation} \label{PRmatrix}
P_\PR = \alpha P + (1-\alpha)\v \e^T,
\end{equation}
where $\e$ is the vector of all ones. Such PageRank transition matrix $P_\PR$ models a new random walk where one takes a step according to the initial Markov chain with probability $\alpha$, and with probability $1-\alpha$ randomly jumps to node $i$ according to the fixed teleportation probability $v_i>0$.
	Note that, as $\v \in \Omega_+$, for any $0\leq\alpha<1$ the {P}age{R}ank matrix $P_\PR$ is irreducible and thus, by the Perron-Frobenius theorem, there exists a unique positive eigenvector $\s \in \Omega_+$ such that  $P_\PR\s = \s$ \cite{berkhin2005survey,tudisco2015complex,varga2009matrix}.
The analogy with \eqref{prob:mpr} essentially follows by \eqref{eq:Z-eig}. In fact, it is not difficult to observe that the proposed multilinear {P}age{R}ank, solution of \eqref{prob:mpr}, coincides with a nonnegative $Z$-eigenvector of the {P}age{R}ank transition tensor $\mathcal P_\PR$, i.e.\ it solves \eqref{eq:Z-eig}. 

More precisely, given a stochastic tensor $\P$ describing the transition probabilities of a higher-order random walk on a graph with $n$ nodes, consider the following {P}age{R}ank transition tensor
\begin{equation}
	\P_\PR:=\alpha \P+(1-\alpha)\V,
\end{equation}
where, given the teleportation vector $\v \in \Omega_+$ and the all-ones vector $\e$,  $\V$ is the rank-one positive tensor $\V = \v \otimes \e \otimes \dots \otimes  \e$, with entries $\V_{i_1, \dots,i_m}=v_{i_1}$.
Now, for any $\s \in \Omega$ we have  $\V\s^{m-1} = (\v \otimes \e \otimes \dots \otimes  \e)\s^{m-1} = (\e^T \s)^{m-1}\v = \v$ and thus it holds
\begin{equation}\label{eq:pr-action}
\P_\PR{\s}^{m-1}=\alpha \P{\s}^{m-1}+(1-\alpha)\V{\s}^{m-1}=\alpha \P{\s}^{m-1}+(1-\alpha)\v ,
\end{equation}
showing that $ \s\in\Omega$ is a fixed-point of \eqref{prob:mpr} if and only if $\P_\PR  \s^{m-1}= \s$.

This tensor eigenvector formulation unveils an elegant conceptual analogy with the matrix case which, however, comes with several fundamental differences.
Tensor eigenvectors can be defined in a number of different ways, including the $Z$-eigenvector formulation we are concerned with in this work. The Perron-Frobenius theory for nonnegative tensors and the numerous corresponding spectral problems is still an active field of research (see for instance \cite{gautier2018unifying} and the references therein).  Several drawbacks arise due  the nonlinearity introduced by the additional dimensions and the Perron-Frobenius theorems developed so far show many crucial differences with respect to the matrix case.
For example, unlike the stochastic matrix case, the irreducibility of the stochastic tensor $\P_\PR$ is not enough to ensure the uniqueness of nonnegative fixed points of \eqref{prob:mpr}, whereas existence is guaranteed in both cases by the Brower's fixed-point theorem \cite{ciarlet2013linear}. {This is shown by the following example, borrowed from \cite{saburov} (see \cite{chang2013uniqueness} for another example):} 

\vspace{0.2cm}
\noindent {
 \textbf{Example\cite{saburov}} Consider the stochastic tensor
\begin{equation*}
\begin{split}
&	\mathcal{P}(:,:,1)=\begin{bmatrix}
	 \frac{232873}{319300}& \frac{7}{10}  &  \frac{3}{10}\\ \\
	 \frac{27}{100}& \frac{470171}{2 \times 814300} &  \frac{378421}{2 \times 407150}\\ \\
	  \frac{54}{79825}& \frac{18409}{2 \times 814300} & \frac{191589}{2 \times 407150 } 
	\end{bmatrix},  \quad 
	\mathcal{P}(:,:,2)=\begin{bmatrix}
	\frac{7}{10} & \frac{4717}{10300}&  \frac{1}{100}\\ \\
	\frac{470171}{2 \times 814300} & \frac{1}{2} & \frac{158157}{2 \times 814300}\\ \\
	\frac{18409}{2 \times 814300} & \frac{433}{10300}&  \frac{1454157}{2 \times 814300}
\end{bmatrix},  \\
&	\mathcal{P}(:,:,3)=\begin{bmatrix}
	 \frac{3}{10}& \frac{1}{100}& \frac{207}{63860} \\ \\
	\frac{378421}{2 \times 407150}&\frac{158157}{2 \times 814300} & \frac{3}{20}\\ \\
	\frac{191589}{2 \times 407150}& \frac{1454157}{2 \times 814300} & \frac{27037}{31930}
\end{bmatrix}.
\end{split}
\end{equation*}
Note that such tensor is entry-wise positive, thus it is irreducible. However, it is such that $\mathcal{P}\mathbf{x}_i^{2}=\mathbf{x}_i$ for $i=1,2,3$, where $\mathbf{x}_1=[0.1,0.2,0.7]^T$, $\mathbf{x}_2=[0.4,0.3,0.3]^T$ and $\mathbf{x}_3=[0.59,0.31,0.1]^T$. }

\vspace{0.2cm}
{
More precisely, for the  case of $Z$-eigenvector and, specifically, for the multilinear {P}age{R}ank tensor $\mathcal P_\PR$, we recall  the following two existence and uniqueness results.
\begin{theorem}[\cite{li2014limiting}]\label{theo:existence}
	If $\A$ is a real  $n$ dimensional stochastic tensor of order $m$, then there exists a nonnegative vector $\s \in \Omega $ such that $\A {\s}^{m-1}= {\s}$. In particular, if $\A$ is irreducible, then ${\s}$ is positive.
\end{theorem}
Observe that if $\v \in \Omega_+$, then $\P_\PR$ is stochastic and irreducible. In fact,
$$
\sum_{i_1=1}^n (\P_\PR)_{i_1,\dots,i_m} = \alpha \sum_{i_1=1}^n p_{i_1,\dots,i_m} + (1-\alpha)\sum_{i_1=1}^n v_{i_1} = 1
$$
and, for any $I\subset V=\{1, \dots,n\}$, we have
$$
(\P_\PR)_{i_1,\dots,i_m} \geq (1-\alpha)v_{i_1}>0
$$
for all $i_1\in I$ and $i_2, \dots, i_m\notin I$.
Hence, Theorem \ref{theo:existence} ensures existence of at least one solution of \eqref{prob:mpr}, namely  there exists $\s\in \Omega$ such that $\P_\PR\s^{m-1}=\s$ and this solution is positive if $\v \in \Omega_+$.
Uniqueness, however, is guaranteed only under somewhat restrictive conditions on $\alpha$. For example, the following result holds (see \cite{fasino2019higher} for more information on uniqueness of fixed-point for generic stochastic tensors and  \cite{li2017uniqueness} for the specific case of the multilinear {P}age{R}ank):
\begin{theorem}[\cite{gleich2015}]\label{theo:existence_uniquness}\footnote{Theorem 1 of the previous \texttt{arXiv} version erroneously states the uniqueness for every $\alpha \in [0,1)$. } Let $\P$ be a $n$ dimensional stochastic tensor of order $m$ and let $\v$ be a nonnegative teleportation vector. Then the multilinear {P}age{R}ank equation
	\begin{equation*}
	\P_\PR \s^{m-1} =\alpha \P{\s}^{m-1}+(1-\alpha)\v =\s
	\end{equation*}
	has a unique solution if $\alpha < (m-1)^{-1}$.
\end{theorem}}

As in the single-order case, we are particularly interested in the case $\alpha\approx 1$. This is because one of the main reasons for introducing the rank-one  perturbation $(1-\alpha)\mathcal V$ is exactly to ensure irreducibility of $\P_\PR$ (and thus existence  of a positive solution, in this case). However, since we are ideally interested in the stationary solution of the original higher-order chain,  the most interesting and relevant problem settings require a very small random surfing perturbation $(1-\alpha)\mathcal V$. In fact, the PageRank solution largely depends on the parameter $\alpha$, { as the following bound shows (see also \cite{fasino2019higher})} 

\begin{proposition}
Consider $\P_\PR(\alpha):= \alpha \P+(1-\alpha)\V$, $\P_\PR(\beta):= \beta \P+(1-\beta)\V$ with $\alpha, \beta \in (0,1)$, and let $\s_\alpha$ and $\s_\beta$ be  corresponding  PageRank solutions. Then, if $0<\alpha<1/(m-1)$, we have
$$
\|\s_\alpha - \s_\beta\|_1 \leq \frac{2\, |\beta-\alpha|}{1-\alpha (m-1)}.
$$
\end{proposition}
\begin{proof}
By adding and subtracting $\alpha \P \s_{\beta}^{m-1}$ from $\|\s_{\alpha}-\s_{\beta}\|_1 = \|\P_\PR(\alpha)\s_{\alpha}^{m-1}-\P_\PR(\beta)\s_{\beta}^{m-1} \|_1$ and using the inequality $\| \P \|_1 \leq 1$ we obtain
\begin{align*}
\begin{aligned}
\|\s_{\alpha}-\s_{\beta}\|_1 &= 
 \| \alpha \P\s_{\alpha}^{m-1}-\beta \P\s_{\beta}^{m-1} +(\beta-\alpha)\v + \alpha \P \s_{\beta}^{m-1} - \alpha \P \s_{\beta}^{m-1} \| \\
 &\leq  \alpha (m-1) \| \P \|_1 \|\s_{\alpha}-\s_{\beta}\|_1+|\alpha-\beta| \, \| \P \|_1 \|\s_{\beta}\|_1 +|\alpha-\beta|\, \|\v\|_1 \\
& \leq \alpha (m-1)   \|\s_{\alpha}-\s_{\beta}\|_1+|\alpha-\beta| \,  \|\s_{\beta}\|_1 +|\alpha-\beta|\, \|\v\|_1\,  .
\end{aligned}
\end{align*}
Thus, as   $\|\s_{\beta}\|_1=\|\v\|_1=1$, we obtain
$$
\big(1-\alpha(m-1)\big)\|\s_{\alpha}-\s_{\beta}\|_1 \leq  2 |\alpha-\beta|
$$ 
and the thesis follows by rearranging terms. 
\end{proof}

On the other hand, as observed in \cite{gleich2015} and as highlighted by the numerical experiments we will present in Section \ref{sec:numerical_results}, solving the multilinear {P}age{R}ank problem as the fixed point of $f$ defined in \eqref{prob:mpr} becomes more difficult when $\alpha$ gets closer to one. This is largely due to the fact that the contractivity of the map $f$ decreases when $\alpha$ increases. This phenomenon is shown for example  in \cite{gautier2018contractivity}, where the contractivity of the map $\s \mapsto \A\s^{m-1}$ is analyzed in terms of the entries of $\A$, and thus of $\alpha$ in our setting $\A =\P_\PR$. Similar contractivity conditions on $\alpha$ are also proved in \cite{gleich2015,li2017uniqueness}, for the specific PageRank case.

In the following section we review two  fixed-point iterative techniques for computing a solution of \eqref{prob:mpr}. Then, in Section \ref{sec:extr_framework}, we introduce our new acceleration framework based on the Simplified Topological $\varepsilon$-algorithm.
\subsection{Power methods for the multilinear {P}age{R}ank} {Due to their simplicity of implementation and their cheap storage requirements, power method-type iterations are typically the preferred choice for large scale problems. However, in the higher-order setting,  their convergence behavior  may be particularly unfavorable. In this section we review the shifted higher-order power method (an ad-hoc version of the SS-HOPM method originally introduced in \cite{kolda2011shifted}) and the  inner-outer method \cite{gleich2015} for the computation of the  multilinear {P}age{R}ank vector and their convergence properties. We then introduce a restarted extrapolation framework, based on the  simplified topological $\varepsilon$-algorithm \cite{brezinski2014simplified,brezinski2017simplified}, that allows to strongly improve the convergence rates of these methods, at a marginal additional cost per iteration. }

\subsubsection{The (higher-order) power method} \label{subsec:power_method}
The power method is one of the best used techniques for the computation of extreme eigenvectors of matrices, in particular this is the standard choice for the computation of the (single-order) {P}age{R}ank vector. This iterative method is essentially a fixed-point iteration scheme and it has been extended to the tensor setting following this interpretation. Precisely, given an initial guess $\s_0$ and the stochastic tensor $\A$, the higher-order power method defines the sequence
\begin{equation}\label{eq:hopm}
\s_{\ell+1} = \A\s_{\ell}^{m-1} + \gamma \s_\ell, \qquad \ell=0,1,\dots
\end{equation}
where the real shifting parameter $\gamma$ can be used to tune the behavior of the method when convergence is not reached for $\gamma=0$ \cite{kolda2011shifted}.

However, the convergence behavior of the method changes significantly when moving from the matrix to the tensor setting. {In the matrix case, convergence to the stationary distribution is ensured for any irreducible aperiodic chain \cite{durrett2012essentials}. More precisely, if $A$ is an irreducible and aperiodic stochastic matrix \cite{varga2009matrix}, then  for any positive vector $\s_0\in\Omega_+$ we are guaranteed that the sequence $\s_{\ell+1}=A\s_{\ell}$, $\ell=0,1,\dots$ converges  to the unique $\bar \s\in \Omega_+$ such that $A\bar \s = \bar \s$. 
%
%
%
Note that we have, equivalently, $\s_{\ell+1} = A^{\ell}\s_0$ and actually, for irreducible and aperiodic chains, the whole matrix sequence $A^\ell$ converges to the rank one matrix $\bar \s\e^T$ (see \cite{tudisco2015complex}, e.g.).  The situation is different in the tensor settings. 
{Due to the nonlinearity introduced by the additional modes, assuming that the stochastic tensor $\mathcal{A}$ is irreducible or aperiodic \cite{bozorgmanesh2016convergence}, is not enough to ensure the convergence of the higher-order power method $\s_{\ell+1} = \mathcal A\s_{\ell}^{m-1}$ for an arbitrary starting point $\s_0 \in \Omega_+$}. 

Note that  iteration   \eqref{eq:hopm} does not preserve the stochasticity of the iterates. To circumvent this problem,  for the specific case of the multilinear {P}age{R}ank $\A = \P_\PR$, the fixed-point equation \eqref{prob:mpr} can be equivalently reformulated as
\begin{equation*}
(1+\gamma)\s=[\alpha\P\s ^{m-1}+(1-\alpha)\v]+\gamma \s \, .
\end{equation*}
Based on this observation,  the following alternative method, called the \textit{shifted fixed-point method} (in short SFPM), is proposed  in \cite{gleich2015} in place of \eqref{eq:hopm},
\begin{equation} \label{methd:SSHOPM}
\s_{\ell+1}=\frac{\alpha}{1+\gamma}\P\s_{\ell}^{m-1}+\frac{1-\alpha}{1+\gamma}\v +\frac{\gamma}{1+\gamma}\s_\ell.
\end{equation}
If $\s_0 \in \Omega$, the iterations produced by \eqref{methd:SSHOPM} are guaranteed to be stochastic and the following result holds:
\begin{theorem}[\cite{gleich2015}] \label{theo:SSHOP}
	Let $\P$ be an order-$m$ cubical stochastic tensor, let $\v\in \Omega_+$ and $\s_0\in\Omega$ be stochastic vectors, and let $\alpha < 1/(m-1)$. The shifted fixed-point method \eqref{methd:SSHOPM} produces stochastic iterations, converges to the unique positive solution $\s$ of the multilinear {P}age{R}ank problem and
	\begin{equation*}
	\|\s_{\ell}-\s\|_1 \leq 2 \Big(\frac{\alpha(m-1)+\gamma}{1+\gamma}\Big)^\ell.
	\end{equation*}
\end{theorem}

When $\alpha > 1/(m-1)$  it is recommended to use the shifted iteration with $\gamma \geq (m-1)/2$ \cite{gleich2015}.
\subsubsection{The inner-outer method} \label{subsec:inner_outer}
The inner-outer method for the standard {P}age{R}ank problem was proposed in \cite{gleich2010inner}. An extension to the multilinear {P}age{R}ank problem is then proposed in \cite{gleich2015}. This is an implicit nonlinear iteration scheme that uses the multilinear {P}age{R}ank in the convergent regime as a subroutine. In order to derive this method one first rearranges the fixed-point iteration $\A\s^{m-1}=\s$ into
%
\begin{equation} \label{eq:inner_outer_rearrangement}
\s =	\frac{\alpha}{m-1}\A\s^{m-1}+\Big(1-\frac{\alpha}{m-1}\Big)\s .
\end{equation}
Then, using \eqref{eq:inner_outer_rearrangement}, the following nonlinear implicit iteration scheme arises:
\begin{equation}\label{eq:in_out_it}
\s_{\ell+1}=\frac{\alpha}{m-1}\A\s_{\ell+1}^{m-1}+\Big(1-\frac{\alpha}{m-1}\Big)\s_\ell
\end{equation}

The iterative scheme in \eqref{eq:in_out_it} requires, at each step, the solution of a multilinear {P}age{R}ank problem  which involves $\A=\P_\PR$, $\alpha/(m-1)$ and $\s_\ell$. More precisely, $\s_{\ell+1}$ in \eqref{eq:in_out_it} is the solution of the following fixed-point problem:
\begin{equation*}
\s= \widetilde {\P}_\PR\s^{m-1} = \widetilde{\alpha}\P_\PR\s^{m-1}+ (1-\widetilde{\alpha})\s_\ell,
\end{equation*}
where $\widetilde\P_\PR = \widetilde \alpha \P_\PR + (1-\widetilde \alpha)\mathcal S_\ell$, $\widetilde{\alpha}=\alpha/(m-1)$ and $\mathcal S_\ell$ is the rank-one tensor $(\mathcal S_\ell)_{i_1,\dots, i_m}= (\s_\ell)_{i_1}$. Note that, since $\alpha<1$, we have $\widetilde{\alpha}<(m-1)^{-1}$. This reformulation unveils the analogy with \eqref{eq:pr-action} and thus,  using Theorem \ref{theo:existence_uniquness},  problem \eqref{eq:in_out_it} has a unique solution  for any $\ell=0,1,\dots$ and its computation can be addressed using the (higher-order) power method with guarantee of convergence. The following result holds:
\begin{theorem}[\cite{gleich2015}] \label{theo:inner_outer}
	Let $\P$ be an order-$m$ stochastic tensor, let $\v \in \Omega_+$, $\s_0 \in \Omega$,  let $\alpha<1/m-1$, and $\mathcal{P}_\PR=\alpha \mathcal{P}+(1-\alpha)\mathcal{V}$. The inner-outer multilinear PageRank method as defined in \eqref{eq:in_out_it}, produces stochastic iterations, converges to the unique positive solution $\s$ of the multilinear {P}age{R}ank problem and
	\begin{equation*}
	\|\s_\ell-\s \|_1 \leq 2\Big(\frac{1-\alpha/(m-1)}{1-\alpha^2}\Big)^\ell.
	\end{equation*}
\end{theorem}
\noindent In \cite{gleich2015}, the inner-outer  iteration \eqref{eq:in_out_it} has been proved to be one of the most effective methods to seek a solution of the multilinear {P}age{R}ank problem even when $\alpha>{1}/({m-1})$. We will see in Section \ref{subsection:extrapolated_innerouter} that suitably coupling  extrapolation techniques with the implicit fixed-point iteration \eqref{eq:in_out_it} improves the effectiveness and efficiency of this iterative method.

%
%
%
%
\section{Extrapolation of  power method and fixed-point iterations} \label{sec:extr_framework}
In the past years, methods for accelerating the convergence of a  sequence $({\bf s}_{\ell})$ of objects in a vector space (e.g. scalars, vector,
matrices) have been developed and successfully applied to a variety of problems, such as the solution of linear and nonlinear
systems, matrix eigenvalue problems, the computation of matrix functions, the solution of integral equations, and many others
\cite{bouhamidi2011extrapolated,brezinski2013extrapolationbook,brezinski2017simplified,brezinski1998extrapolation}.

In some situations, these methods are ad-hoc modifications of the methods that
produced the corresponding original sequences.
However, often, the process that generates $({\bf s}_{\ell})$  is too cumbersome for this approach
to be practical. Thus, a  common and successful solution is to transform the original sequence
$({\bf s}_{\ell})$  into a new sequence $({\bf t}_{\ell})$  by means
of a {\it sequence transformation} $T$, which, under some assumptions,
converges faster to the limit or, in the case of diverging sequences,  the antilimit ${\bf s}$  of $({\bf s}_{\ell})$.

The idea behind a sequence transformation is to assume that the
sequence  to be transformed behaves like a model sequence whose limit $\bf s$ (or antilimit in the case of divergence) can be exactly computed by a finite algebraic process.
The set ${\bf {\cal K}}_T$ of these model sequences is called the kernel of the transformation. If the sequence
$({\bf s}_{\ell})$ belongs to the kernel ${\bf {\cal K}}_T$, then the transformed sequence ``converges in one step'', namely ${\bf t}_\ell={\bf s}$ for all $\ell$. Instead, if the sequence $({\bf s}_{\ell})$ does not belong to the kernel but it is close enough to it, then the  sequence transformation often produces a remarkable convergence acceleration.

Among the existing sequence transformations and acceleration methods (also called {\it extrapolation methods}), the Shanks transformation
\cite{shanks1955non} is arguably the best
all-purpose method for accelerating the convergence of a sequence.
The kernel of the vector Shanks' transformation can be represented by the difference equation
\begin{equation}
	a_0({\bf s}_\ell-{\bf s})+\cdots +a_k({\bf s}_{\ell+k}-{\bf s})=0,
\quad \ell=0,1,\ldots
\label{kernelnew}
	\end{equation}
with $ a_i \in \mathbb R, a_0a_k \neq 0$ and $a_0+\cdots+a_k \neq 0$. 
Then, assuming without loss of generality that $a_0+\cdots+a_k=1$, the Shanks extrapolated sequence ${\bf t_\ell}= {\e}_k({\bf s}_\ell)$ can be written as
$$
{\e}_k({\bf s}_\ell) =	a_0{\bf s}_\ell+\cdots +a_k{\bf s}_{\ell+k}.
$$
{Several 
vector sequence transformations
based on  such a kernel 
exist and have
been introduced and studied by various authors (the $\varepsilon-algorithms$, the MMPE, the MPE, the RRE, the E-algorithm, etc., see \cite{brezinski2013extrapolationbook} for a review).}
For all of them,
the following theorem  holds:
\begin{theorem} \label{THkernel}
	If there exist $\;a_0, \ldots, a_k$ with $a_0a_k \neq 0$ and $a_0+ \cdots + a_k \neq 0$ such that, for all $\ell$,
	$$
	a_0({\bf s}_\ell-{\bf s})+\cdots +a_k({\bf s}_{\ell+k}-{\bf s})=0,
	$$
	then, for all $\ell$, ${\e}_k({\bf s}_\ell)={\bf s}$.
\end{theorem}
That is, if the sequence $({\bf s}_{\ell})$ belongs to the kernel ${\bf {\cal K}}_T$, then it is transformed
into a constant sequence whose terms are all equal to the limit (or the antilimit)   $\bf s$.

The Topological Shanks transformations are arguably the most general Shanks-based transformations. They have a kernel of the form \eqref{kernelnew}, and
they can be applied to sequences of elements of an arbitrary  vector space. They
can be recursively implemented by the Topological Epsilon Algorithms, in short TEA's \cite{brezinski1975generalisations}.
They allow to consider not only sequences of scalars or vectors of ${\mathbb R}^n$, but also matrices or tensors.
However, their main drawback is the difficulty of implementing the related algorithms.
Recently, simplified versions of these algorithms, called the {Simplified Topological Epsilon Algorithms}
(STEA's), have been introduced. These simplified algorithms have three main advantages with respect to the original algorithms:
the numerical stability can be sensibly improved, the rules defining the extrapolated sequence are simpler than the original ones and
the cost is reduced both in terms of memory allocation and in terms of operations to be performed. In the next Section \ref{sec:topological_epsilon_algorithm}, we briefly review  these
topological  Shanks transformations and the STEA algorithms. For further details, see
 \cite{brezinski2014simplified,brezinski2017simplified}. 

The use of Shanks-type acceleration techniques has proved to be very effective in order to speed up the computation of  the {P}age{R}ank
vector of a graph \cite{brezinski2006pagerank,brezinski2008rational,brezinski2005extrapolation,kamvar2003extrapolation}. To briefly review why, note that,  given the PageRank matrix \eqref{PRmatrix}, with $\alpha \in [0,1)$, we know that the power iterations
$\s_{\ell+1} = P_\PR \s_{\ell} = P^\ell_\PR \v, \ell = 0,1, \ldots$, with
$\bf \s_0 = \v$,
converge to the unique vector $\s = P_\PR \s$ and this vector can be expressed explicitly as a polynomial of $P_\PR$. In fact, if $\Pi_p$ is the minimal polynomial of
$P_\PR$ for the vector $\v$ (with $p \le n$ being its degree), since $P_\PR$ has an eigenvalue equal to $1$, we have
$  \Pi_p(\lambda)=(\lambda-1)Q_{p-1}(\lambda)$, where $Q_{p-1}$ is a polynomial
of degree $p-1$.
Thus
$ \s= Q_{p-1}(P_\PR)\v$. At the same time we have
$ \s= P^\ell_\PR \s = Q_{p-1}(P_\PR) \, \s_\ell$
and so the vectors $\bf \s_\ell - \s$ satisfy a difference equation {whose form is exactly the one of} 
the Shanks kernel \eqref{kernelnew}, that is
$$
a_0({\bf s}_\ell-{\bf s})+\cdots +a_{p-1}({\bf s}_{\ell+p-1}-{\bf s})=0,
$$

for some $a_0, \dots, a_{p-1}\in \mathbb R$. Thus it holds ${\e}_{p-1}({\bf s}_\ell)={\bf s}, \forall \ell$. If we take
$k-1\ll p-1$, and if we consider a polynomial $\tilde Q_{k-1}$ of degree $k-1$ approximating the
polynomial $Q_{p-1}$, the particular structure of the PageRank matrix ensures that the sequence  ${\e}_{k-1}({\bf s}_\ell)$ computed, for instance,
by a STEA algorithm (a) is a good approximation of $\s$ and (b) converges much faster than the original sequence produced
by the power method.

Since the multilinear PageRank problem is a generalization of the PageRank model (and it reduces to it when $m=2$ \cite{gleich2015}), we test here the performance of the same class of Shanks-based algorithms on it. 

The numerical results show a remarkable speed up both in terms of number of iterations and in terms of execution time and, to our opinion, this  promising converging behavior is largely due to the several similarities between the case $m=2$ and the higher-order case $m>2$. 

The extrapolation algorithms can be coupled with a restarting technique which is particularly suited for fixed-point problems and that roughly proceeds as follows.
Assume that we have to find a fixed-point $\bf s$ of a mapping $F: {\mathbb R}^n \rightarrow
{\mathbb R}^n$.  We compute a certain number of basic iterates ${{\bf s}_{\ell+1}}= F({\bf s}_\ell)$ from a given ${\bf s}_0$.
Then we apply the extrapolation algorithm to them, and we restart the
basic iterates from the computed extrapolated term.
The advantage of this approach is that, under suitable regularity assumptions on $F$ and if the number of extrapolation steps $k$ is large enough, the sequence generated in this way  converges  quadratically to the fixed point of $F$ \cite{le1992quadratic}.
 The details of this restarted procedure and the application to  the specific multilinear {P}age{R}ank setting is described in  Section \ref{sec:stea_for_powermethod}. 

\subsection{The topological Shanks transformations and algorithms} \label{sec:topological_epsilon_algorithm}
Consider a sequence of elements $({\bf s}_\ell)$ of a topological vector space $E$  with $\lim_{\ell \to +\infty } {\bf s}_\ell=\bf{s}$.
The so-called first and second Topological Shanks Transformations, starting from  the original
sequence and given an arbitrary element of the dual space $\y \in E^*$, respectively produce two new sequences $(\widehat{\e}_k({\bf s}_\ell))$ and $(\widetilde{\e}_k({\bf s}_\ell))$
where each term uses $2k+1$ vectors and has the form
\begin{align}
& \mbox{First transformation~~} \widehat{\e}_k({\bf s}_\ell)=a_0^{(\ell,k)}{\bf s}_\ell+\cdots+a_k^{(\ell,k)}{\bf s}_{\ell+k}, \label{eq:shanks_transformations1}\\
& \mbox{Second transformation~~} \widetilde{\e}_k({\bf s}_\ell)=a_0^{(\ell,k)}{\bf s}_{\ell+k}+\cdots+a_k^{(\ell,k)}{\bf s}_{\ell+2k}, \label{eq:shanks_transformations2}
\end{align}
where the $a_i^{(\ell,k)}$ are the solutions of the system
\begin{equation} \label{eq:linear_system_neq_sequence}
\left\{
\begin{array}{c}
a_0^{(\ell,k)}+ \cdots + a_k^{(\ell,k)}=1, \\
a_0^{(\ell,k)}\langle \y, \Delta {\bf s}_{\ell + i}\rangle+ \cdots + a_k^{(\ell,k)}\langle\y, \Delta {\bf s}_{\ell+k+i}\rangle=0, \;\;\; i=0, \dots, 2k-1,
\end{array}
\right.
\end{equation}
with $\Delta \s_i:=\s_{i+1}-\s_i$ and where $\langle \cdot , \cdot \rangle$ is the duality product, which reduces to the usual inner product between two vectors when $E = \mathbb{R}^n$. The first relation is a normalization condition
that does not restrict the generality.

For both of these transformations,  Theorem \ref{THkernel} holds and
if the sequence satisfies \eqref{kernelnew}, then
$$
\widehat{\e}_k({\bf s}_\ell)={\bf s} \quad {\rm and} \quad \widetilde{\e}_k({\bf s}_\ell)={\bf s}.
$$

The topological $\varepsilon$-transformations can be used also when the original sequence does not belong to
the kernel, i.e.\ it does not satisfy \eqref{kernelnew}. In this case the coefficients
$a_i$ depend on $\ell$ and $k$, and this dependence is  emphasized by the upper indices in \eqref{eq:shanks_transformations1} and
\eqref{eq:shanks_transformations2}.

Recursive algorithms allow to compute the terms $\widehat{\e}_k({\bf s}_\ell)$, $\widetilde{\e}_k({\bf s}_\ell)$ of
the new sequences without solving explicitly the linear system \eqref{eq:linear_system_neq_sequence}.	
Moreover, the simplified forms of the topological $\varepsilon$-algorithms \cite{brezinski2014simplified, brezinski2017simplified}, allow an implementation  that avoids the manipulation of the elements of $E^*$
(in contrast to the topological $\varepsilon$-algorithms \cite{brezinski1975generalisations}), since the linear
functional $\y$ is applied only to the terms of the initial sequence $({\bf s}_\ell)$ and is not used in the recursive rule.
The simplified algorithms implementing the first transformation
\eqref{eq:shanks_transformations1} and the second one \eqref{eq:shanks_transformations2} are usually denoted by STEA1 and STEA2, respectively. In this work we focus mostly on STEA2 as it requires less memory storage. Moreover, among the four  existing equivalent updating rules, we use the one we observed being the most effective, defined by 
\begin{equation}
{\widetilde{\mathbf{\varepsilon}}}^{(\ell)}_{2k+2}={\widetilde{\mathbf{\varepsilon}}}^{(\ell+1)}_{2k}+
{\frac{\epsilon_{2k+2}^{(\ell)}		 -\epsilon_{2k}^{(\ell+1)}}{\epsilon_{2k}^{(\ell+2)}-\epsilon_{2k}^{(\ell+1)}}}({\widetilde{\mathbf{\varepsilon}}}^{(\ell+2)}_{2k}-
{\widetilde{\mathbf{\varepsilon}}}^{(\ell+1)}_{2k}), \;\;k,\ell=0,1,\dots
\end{equation}
where $\widetilde{\mathbf{\varepsilon}}^{(n)}_{0}=\mathbf{s}_\ell \in E$ and the scalar quantities
$\epsilon_{i}^{(j)}$ are computed by the Wynn's scalar $\varepsilon$-algorithm \cite{wynn1956} applied to
${s_{\ell}} = \langle{\bf{ y}}, {{\bf{ s}}}_{\ell}\rangle$. Note that this rule is very simple: it contains only sums and differences between elements of the vector space  and it relies only on three terms of a triangular scheme. Moreover, it has been proved in \cite{brezinski1975generalisations} that the identity
${\widetilde{\mathbf{\varepsilon}}}^{(\ell)}_{2k}=\widetilde{\mathbf{e}}_k(\mathbf{s}_\ell)$  holds.


\subsection{Restarted extrapolation method for multilinear {P}age{R}ank}\label{sec:stea_for_powermethod}

When dealing with fixed-point problems $F({\bf s})={\bf s}$, a common and pertinent choice is to couple the extrapolation method with a restarting technique\cite{brezinski2017simplified}.
If we consider the STEA's algorithms,  the
general restarted method is presented in the following Algorithm~\ref{alg:restarted_method}.
\begin{algorithm}
	\caption{Restarted extrapolation method}\label{alg:restarted_method}
	\begin{algorithmic}[1]
			\State {Choose $2k, \;\mathsf{cycles}\; \in \mathbb{N}, \x_0 \hbox{ and } \y \in E^{*}$}
			\For {$i=0,1, \ldots, \mathsf{cycles}$ (outer iterations)}
			\State Set ${\bf s}_0=\x_i$
            \State Compute ${s_{0}} = \langle{\bf{ y}}, {{\bf{ s}}}_{0}\rangle$
			\For{$\ell=1, \dots, 2k$ (inner iterations)}
			\State	Compute ${\bf s}_\ell=F({\bf s}_{\ell-1})$
			\State	Compute ${s_{\ell}} = \langle{\bf{ y}}, {{\bf{ s}}}_{\ell}\rangle$
			\EndFor
			\State	Apply STEA to ${\bf s}_0, \dots {\bf s}_{2k}$  and ${s}_0, \dots {s}_{2k}$ to compute $\widehat{\e}_{k}({\bf s}_\ell)$ or $\widetilde{\e}_{k}({\bf s}_\ell)$
			\State	Set $\x_{i+1}= \widehat{\e}_{k}({\bf s}_\ell)$ or $\x_{i+1}= \widetilde{\e}_{k}({\bf s}_\ell)$
			\State	Choose ${\bf y} \in E^*$
			\EndFor
		\end{algorithmic}
\end{algorithm}
It is important to remark that, when $k=n$, where $n$ is the dimension of the problem, the sequence $(\widehat{\e}_{k}({\bf s}_\ell))$ (or $( \widetilde{\e}_{k}({\bf s}_\ell))$)
converges quadratically to $\mathbf s$ under suitable regularity assumptions \cite{le1992quadratic}.
All the schemes we took into consideration in Section \ref{sec:theoretical_background}
are of (implicit or explicit) fixed-point iteration type, thus for all of them we consider restarted simplified topological methods. In our case, $F(\bf s_\ell)$ is either the shifted power method  introduced in Section \ref{subsec:power_method}
(see equation \eqref{methd:SSHOPM}) or the inner-outer method of Section \ref{subsec:inner_outer} (see
equation \eqref{eq:in_out_it}), and $E=\mathbb{R}^n$.

Concerning the computational complexity and the storage requirements, in our experimental investigations, we used the public available  software EPSfun Matlab toolbox, na44 package in Netlib
\cite{bmrz2017software} that contains optimized versions of the topological algorithms we used. In particular, the STEA2 algorithm contained in this package is  implemented by using an ascending diagonal technique. The $2k+1$ vectors ${\bf s}_\ell$ are computed together with the extrapolation scheme, and thus only  $k+2$ vectors of dimension $n$ have to be stored in order to  compute $\widetilde{\e}_{k}({\bf s}_\ell)$. The duality products, that in our cases are always inner products between real vectors, are $2k+1$ for each outer cycle.

Thus, the practical implementation and the performance of the methods relies on two key parameter choices: the choice of $k$ and of $\y \in E^*$. As described above, the choice of $k$  is connected to the memory requirement and determines the quality of the speed-up performance. While we can choose $k\approx n$ for relatively small problems, when the dimension $n$ of the problem is large we selected a small  value of $k$  if compared to $n$, as suggested in \cite{brezinski2014simplified,brezinski2017simplified}. This is the case, for instance, of the real-world  examples of Section \ref{sec:real_world_problems}.
Concerning the choice of $\y \in E^*$, this is a well-known critical point in the topological Shanks transformations and it is usually addressed by model-dependent heuristics. In fact, no general theoretical result has been obtained so far concerning the  selection of an optimal $\y \in E^*$, not even for the case $E=\mathbb{R}^n$.
In our examples we have chosen $\y = \widetilde{\e}_{k}({\bf s}_\ell)$, the last extrapolated term. The quality of this choice is supported by the remarkable performance we obtained and by the fact that in all our tests the resulting extrapolated vectors computed with such a choice of $\y$ are always nonnegative and stochastic. 
Nevertheless, in Section \ref{sec:stoc_iter}, we present further preliminary
results concerning this open and debated problem. Our analysis is particularly important for the multilinear {P}age{R}ank problem we are considering as it deals with the existence and the possible computation of a vector $\y$ that guarantees that the extrapolated terms obtained with the topological $\varepsilon$-transformation are stochastic. We present these theoretical results only as a first attempt to study in depth this unsolved problem, but we do not use them in the numerical experiments as it would have prevented us to fully exploit the advantages given by the use of the optimized package EPSfun. We intend to continue our study in forthcoming investigations.

\subsection{Choosing $\y \in \RR^n$ to enforce stochastic extrapolated vectors} \label{sec:stoc_iter}
As it is stated in Theorem \ref{theo:SSHOP} and Theorem \ref{theo:inner_outer}, the shifted power method and the inner-outer
method produce stochastic iterations. In this section, we will prove that for each outer cycle in Algorithm \ref{alg:restarted_method}
there exists a functional $\mathbf{y} \in E^*\;(= \mathbb{R}^n$ in our case) such that the coefficients
$a_0^{(\ell,k)},\;\dots,a_k^{(\ell,k)}$ in \eqref{eq:shanks_transformations1} and \eqref{eq:shanks_transformations2}
can be chosen nonnegative. Note that in this way  we are guaranteed that, using the normalization condition in
\eqref{eq:linear_system_neq_sequence}, the extrapolation procedure at Line 9 of Algorithm \ref{alg:restarted_method}
produces a stochastic vector (i.e.\ a vector with nonnegative entries that sum up to one).	

Let us point out that this section has just a theoretical interest: as it will be clear from the results presented later on in this section,
the vector $\y \in \mathbb{R}^{n}$ must be computed at each outer cycle in Algorithm \ref{alg:restarted_method} through a
computational procedure whose cost depends on the parameter $2k$ and increases accordingly.
In the numerical experiments we prefer to make a choice which does not require further computations and we choose
hence $\y=\widetilde{\e}_{k}({\bf s}_\ell)$ (or $\widehat{\e}_{k}({\bf s}_\ell)$); this is the most updated approximation of the solution at our disposal and, as
the numerical results will show, this choice works very well in practice.

Before proceeding let us define $b_i:=(\y, \Delta {\bf s}_{{\ell}+i})$ for $i=0,\dots,2k-1$. With this notation the matrix of the
linear system \eqref{eq:linear_system_neq_sequence} can be written as
\begin{equation}
S^{(\ell,k)}:= \begin{bmatrix}
1 & \dots & 1 \\
b_0 & \dots & b_{k} \\
\vdots & & \vdots \\
b_{k-1} & \dots & b_{2k-1}
\end{bmatrix}.
\end{equation}

Let us start by stating the following lemma which better explains the notation introduced above:

\begin{lemma}\label{lem:y_choiche}
	If $\Delta {\bf s}_{\ell+i}$,  $i=0,\dots, 2k-1$, are linearly independent vectors and $2k \leq n$, then for any choices
	of $2k$ real numbers $b_i$, there exists a vector $\mathbf{y} \in \mathbb{R}^{n}$ such that
	$b_i=(\mathbf{y}, \Delta {\bf s}_{\ell+i})$, for $i=0,\dots,2k-1$.
\end{lemma}
\begin{proof}
	Consider the QR-decomposition of the matrix $\Delta=[\Delta {\bf s}_{\ell}, \dots, \Delta {\bf s}_{\ell+2k-1} ]$, i.e.,
	$\Delta=QR$ being $Q \in \mathbb{R}^{n \times 2k}$ such that $Q^TQ=I_{2k}$ and $R$ upper triangular of rank $2k$.
	Setting $Q^T\y=\g$ for some $\g \in \mathbb{R}^{2k}$, we have $\y^T\Delta = \g^TR$. Since $R$ is invertible, for any
	choice of the vector $\mathbf {b} \in \mathbb{R}^{2k}$, there exists a unique set of coefficients $\bf g$ solution of the linear
	system $R^T\bf g=\bf b$. The result follows by observing that $2k=\mathrm{rank}(Q^T)=\mathrm{rank}([Q^T|\mathbf{b}])$ for any
	$\mathbf{b} \in \mathbb{R}^{2k}$.
\end{proof}

Note that the lemma above actually provides us a constructive way to compute a vector $\bf y$ such that
$b_i= (\mathbf{y}, \Delta {\bf s}_{\ell+i})$, $i=0,\dots, 2k-1$. Moreover, at the same time, it shows that the explicit
computation of $\mathbf{y}$ is not necessary anymore to prove its existence: the extrapolation process is well defined as soon
as  the coefficients $b_i$ for $i=0 \dots 2k-1$ are chosen.  Observe, finally, that the linear independence of the
vectors $\Delta {\bf s}_{\ell+i}$ for $i=0,\dots, 2k-1$ is not a restrictive hypothesis since it is equivalent to the linear
independence of the vectors $\mathbf{s}_{\ell +i}$ and, if the $\mathbf{s}_{\ell +i}$ are not linear independent, we can always
select a maximal set of linearly independent vectors and use them in place of the $\mathbf{s}_{\ell +i}$. This would not affect  the final result of a Shanks-based
extrapolation process since it essentially is a linear combination (see equations \eqref{eq:shanks_transformations1}
and \eqref{eq:shanks_transformations2}).

Define  $J$ as the matrix with all ones on the anti-diagonal. Note that $J^2=I$, thus we can write the linear system
\eqref{eq:linear_system_neq_sequence}	as
\begin{equation*}
S^{(\ell,k)}JJ\mathbf{a}^{(\ell,k)}=\mathbf{e}_1,
\end{equation*}

where
\begin{equation*}
S^{(\ell,k)}J=\begin{bmatrix}
1 & \dots & 1 \\
b_k & \dots & b_{0} \\
\vdots &  & \vdots \\
b_{2k-1} & \dots & b_{k-1}
\end{bmatrix},
\end{equation*}
and $\mathbf{e}_1=[1, 0,\dots, 0]^T$.
Thus, if we define the Toeplitz matrix $T$ as
\begin{equation}\label{eq:toeplitz_matrix}
T:=\begin{bmatrix}
b_{k-1} & b_{k-2} & \dots & b_0 & {0}\\
b_k & \ddots & \ddots &  & b_{0} \\
\vdots & \ddots & \ddots & \ddots & \vdots \\
\vdots &  & \ddots  & \ddots & b_{k-2} \\
b_{2k-1} & \dots & \dots & b_k & b_{k-1}
\end{bmatrix},
\end{equation}
we can write $S^{(\ell,k)}J$ as
\begin{equation*}
S^{(\ell,k)}J=T+\mathbf{e}_1\mathbf{v}^T
\end{equation*}
being  $\mathbf{v}=[1-b_{k-1},\dots, 1-b_0, 1  ]$. In the next Theorem  \ref{eq:theorem_sufficient_conditions} we give  sufficient
conditions for the $b_i$'s in  order to guarantee that the resulting coefficients $a_0^{(\ell,k)},\;\dots,a_k^{(\ell,k)}$ are all
nonnegative. We assume that the matrix $S^{(\ell,k)}$ of the linear system \eqref{eq:linear_system_neq_sequence} (and hence the
matrix $S^{(\ell,k)}J$) is invertible. To this end, we need one further lemma:
\begin{lemma} \label{lem:I_geq_H}
	If $T$ is invertible, $I \geq T$ (element wise) and $\|-T+1\|<1$ for some matrix norm, then $T^{-1}$  is a nonnegative matrix.
\end{lemma}
\begin{proof}
	Since $T$ is invertible, we have
	\begin{equation}
	T^{-1}=(I-(-T+I))^{-1}=\sum_{j=0}^{+ \infty}(-T+I)^j.
	\end{equation}
	where the  Neumann expansion holds since $\|-T+I\|<1$.
	The result follows using the hypothesis $I \geq T$.
\end{proof}
The following main theorem holds
\begin{theorem} \label{eq:theorem_sufficient_conditions}
	There exists a choice of the coefficients $b_i$ for $i=0,\dots, 2k-1$ such that $a_0^{(\ell,k)},\;\dots,a_k^{(\ell,k)}$ in
	\eqref{eq:linear_system_neq_sequence} are nonnegative.
\end{theorem}
\begin{proof}
	Consider the symmetric Toeplitz matrix $T$ defined in \eqref{eq:toeplitz_matrix},  obtained by choosing $b_{k-1}=1$,
	$b_{k-i}=b_{k-2+i}<0$ for $i=2, \dots, k$, $b_{2k-1}=0,$ and such that $\sum_{i=2}^k|b_{k-i}| <1$. With this settings,
	$T$ is a symmetric positive definite matrix such that $I \geq T$ (element wise) and $\|-T+1\|_2<1$. Moreover,
	$S^{(\ell,k)}J=T+\mathbf{e}_1\mathbf{v}^T$ with $\mathbf{e}_1$ and $\mathbf{v}$ nonnegative; using the Sherman-Morrison
	formula, it holds
	\begin{equation}
	(S^{(\ell,k)}J)^{-1}=T^{-1}-\frac{1}{1+\mathbf{v}^TT^{-1}\mathbf{e}_1}T^{-1}\mathbf{e}_1\mathbf{v}^TT^{-1},
	\end{equation}
	and hence \begin{equation}
	J\begin{bmatrix}
	a_0^{(\ell,k)} \\
	\vdots \\
	a_k^{(\ell,k)}
	\end{bmatrix}=(S^{(\ell,k)}J)^{-1}\mathbf{e}_1=\frac{1}{1+\mathbf{v}^TT^{-1}\mathbf{e}_1}T^{-1}\mathbf{e}_1.
	\end{equation}
	The result follows from Lemma \ref{lem:I_geq_H}  observing that $T^{-1}\mathbf{e}_1$ is nonnegative and that
	$1+\mathbf{v}^TT^{-1}\mathbf{e}_1>0$.
\end{proof}

\section{Numerical results} \label{sec:numerical_results}

In this section we present several numerical experiments to demonstrate the advantages of the extrapolation framework we are proposing. In all our experiments we consider the relative $1$-norm residual
\begin{equation}\label{eq:residual}
r_1(\x)= \frac{\|\P_{\PR}\x^{m-1}-\x\|_1}{\|\x\|_1} \, ,
\end{equation}
evaluated on the current iteration step. We choose the $1$-norm as the generated sequences are stochastic and thus the relative $1$-norm residual boils down to $r_1(\x)=\|\alpha\P\x ^{m-1}+(1-\alpha)\v-\x \|_{1}$.

Sections \ref{sec:problemset1} and \ref{sec:problemset2} are devoted to analyze and highlight  the rate of convergence of
the accelerated sequence and to compare it to the one of the original sequence. To this end, we  run Algorithm \ref{alg:restarted_method} for  a prescribed number of inner and outer iterations (i.e.\ we fix the value of $k$ and $\mathsf{cycles}$)  but without any other stopping criterion. 
Results are shown in Figures \ref{fig:experiment1}--\ref{fig:experiment5_rand_tens}: Figures \ref{fig:experiment1}--\ref{fig:experiment_inout}  compare the convergence behavior of the methods on a number of small size benchmark test problems, whereas Figures \ref{fig:experiment1_rand_tens}--\ref{fig:experiment5_rand_tens} propose an analysis of the methods' performance on 100 random tensors of different moderate sizes via boxplots.

Section \ref{sec:real_world_problems}, finally, is devoted to analyze several real-world datasets of different moderate-to-large sizes. In  this section, in order to compare computational timings, we do not fix an a-priori number of iterations, instead we run  each method until the residual \eqref{eq:residual} is smaller than  $10^{-8}$.

In all the Figures we denote by $\x_{\ell}$ the vectors of the original sequence obtained by the {shifted fixed-point method} (SFPM) or by the inner-outer method (IOM). Whereas, for the restarted extrapolation method (STEA2) we denote by $\widetilde{\x}_{\ell}$ the vectors of the sequence generated by Algorithm \ref{alg:restarted_method}. In the plots of  Figures \ref{fig:experiment1}, \ref{fig:experiment2} and \ref{fig:experiment_inout} we highlight with a circle each restart of the outer loop, i.e.\ the  vector defined at Line 10 of Algorithm \ref{alg:restarted_method}.


The linear functional $\y$ is updated at the end of each outer cycle by choosing $\y=\widetilde{\e}_{k}({\bf s}_\ell)$, without any additional cost (for the first extrapolation step we choose a random vector). 
As previously pointed out, we observed experimentally that the extrapolated vectors obtained in this way are all stochastic.
All the numerical experiments are performed on a laptop running Linux with 16Gb memory and CPU Intel\textsuperscript{\textregistered} Core\texttrademark\ i7-4510U with clock 2.00GHz. The code is written and executed in MATLAB R2015a. For the implementation of the simplified topological $\varepsilon$-algorithm we used the public domain Matlab toolbox EPSfun \cite{bmrz2017software}.

\begin{figure}[p!]
	\centering
	$\,$\hfill\includegraphics[width=0.4\columnwidth]{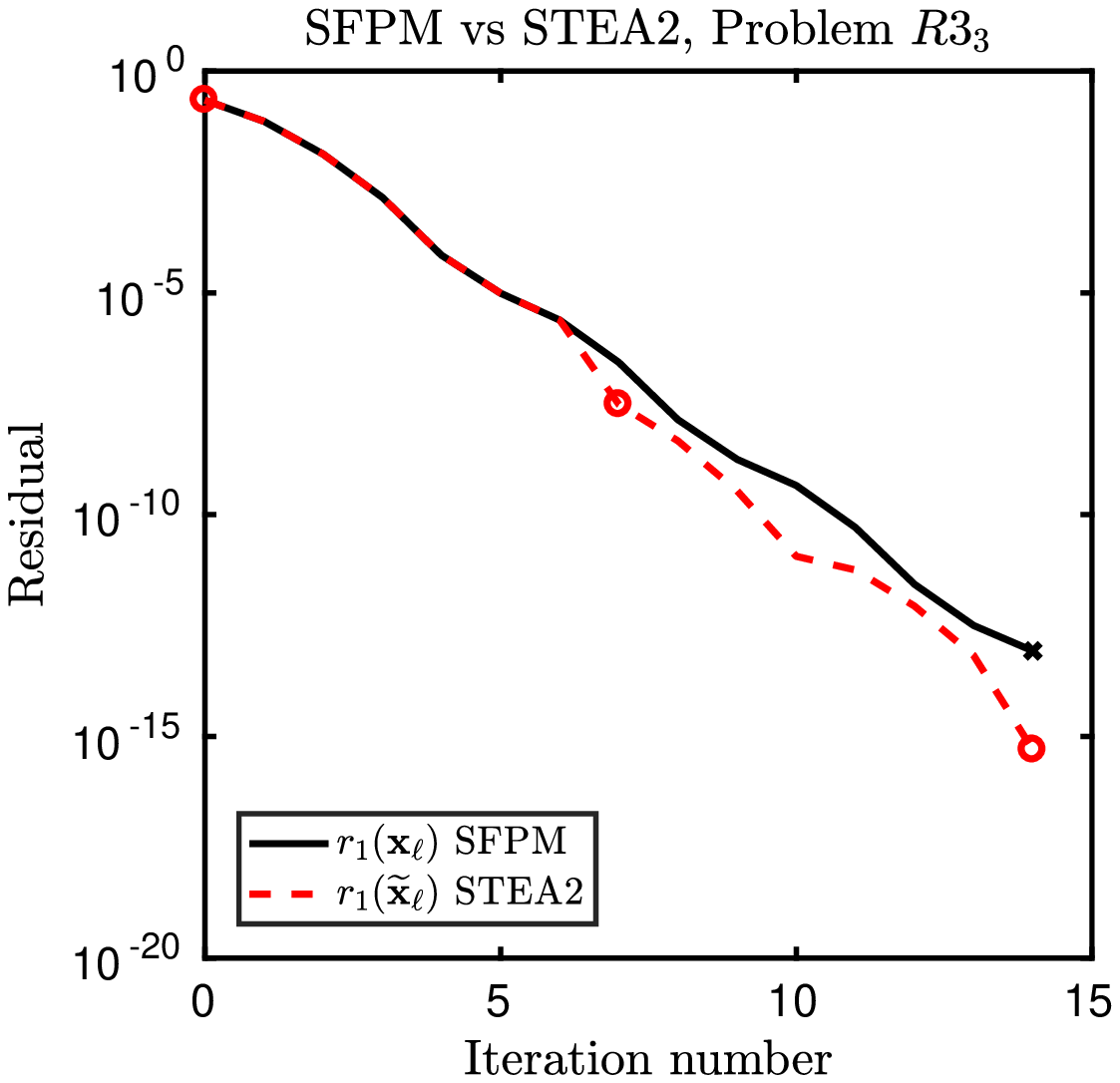} \hfill
	\includegraphics[width=0.4\columnwidth]{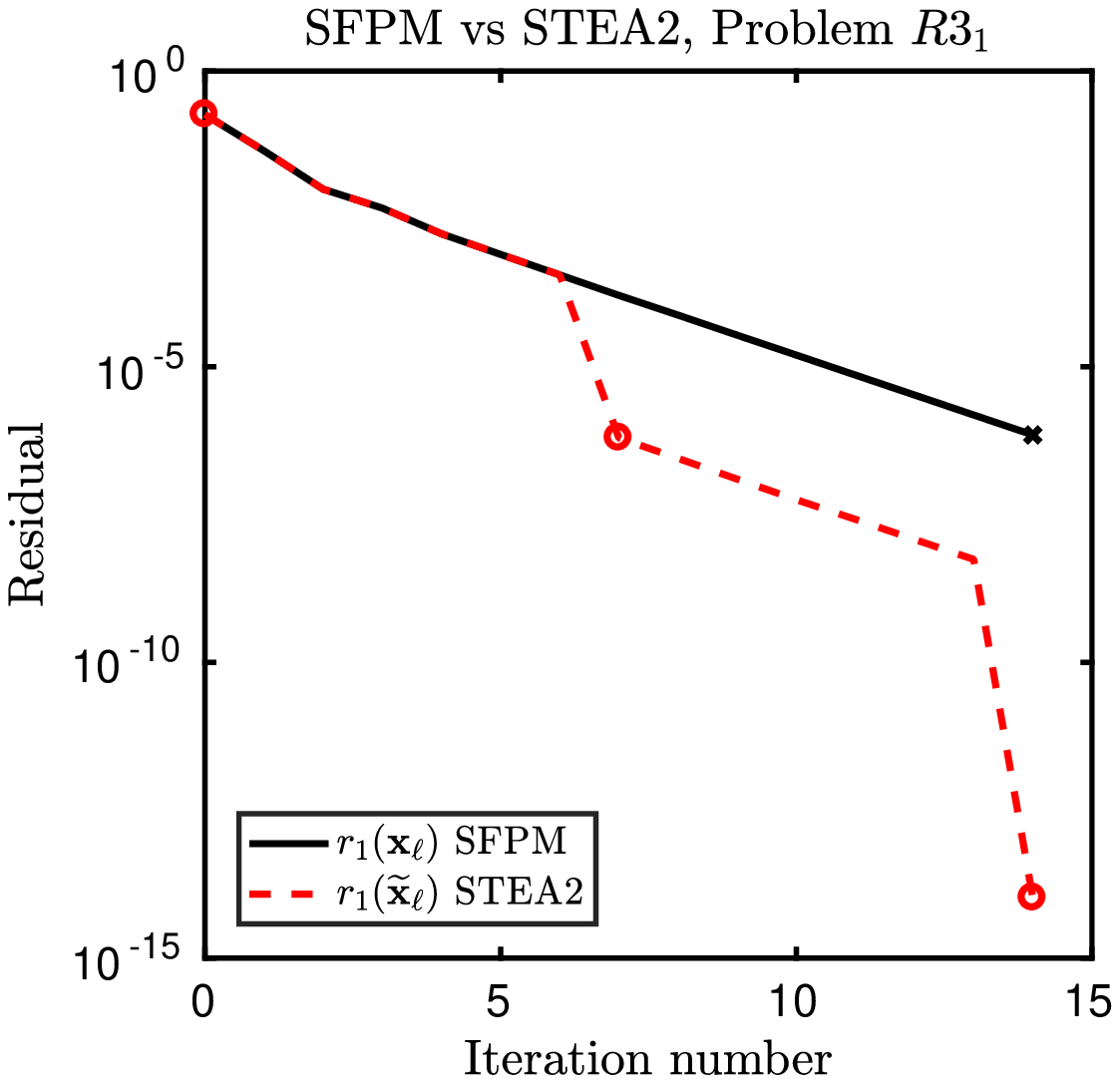} \hfill $\,$\\[1.5em]
	$\,$\hfill\includegraphics[width=0.4\columnwidth]{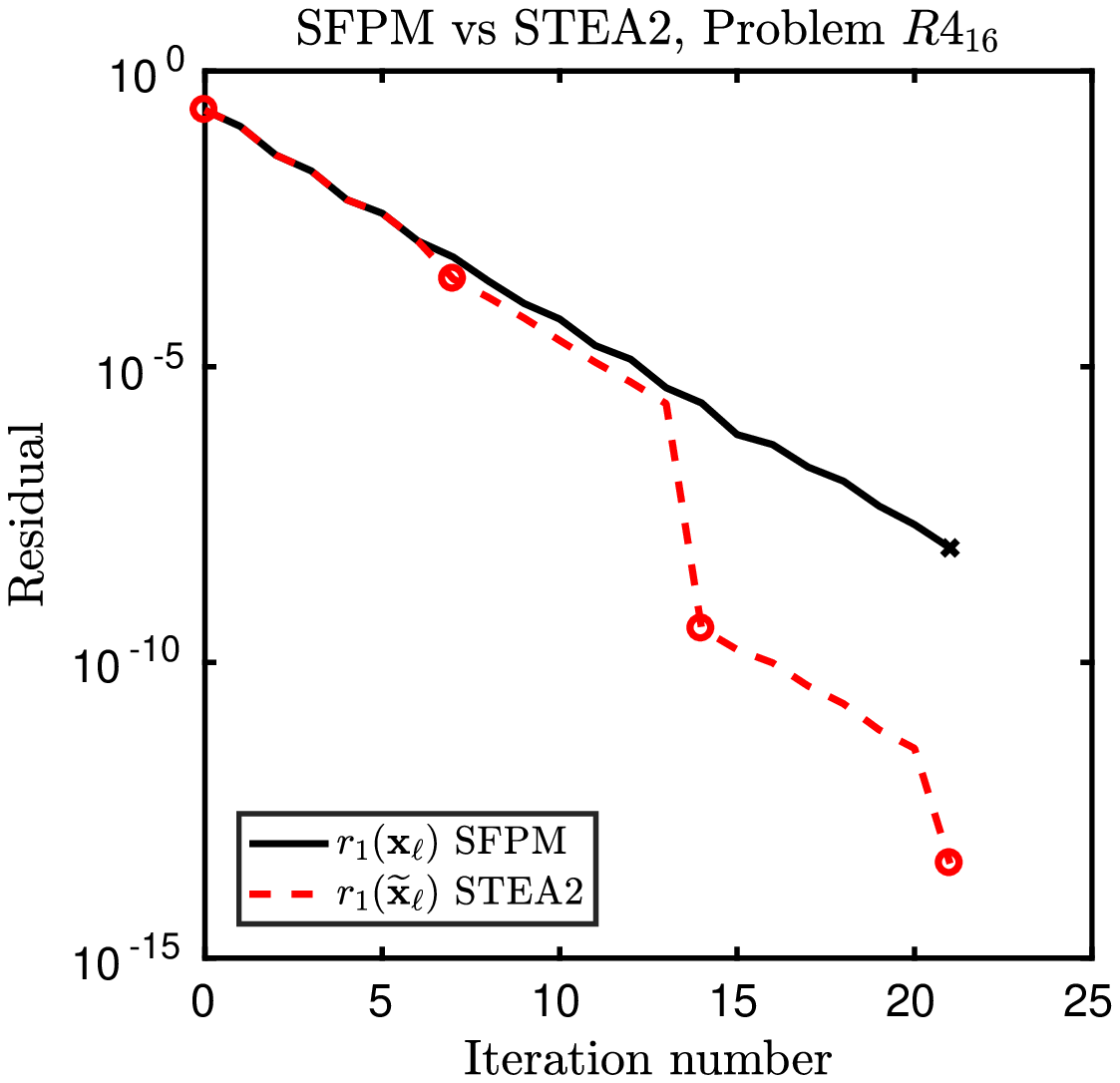} \hfill
	\includegraphics[width=0.4\columnwidth]{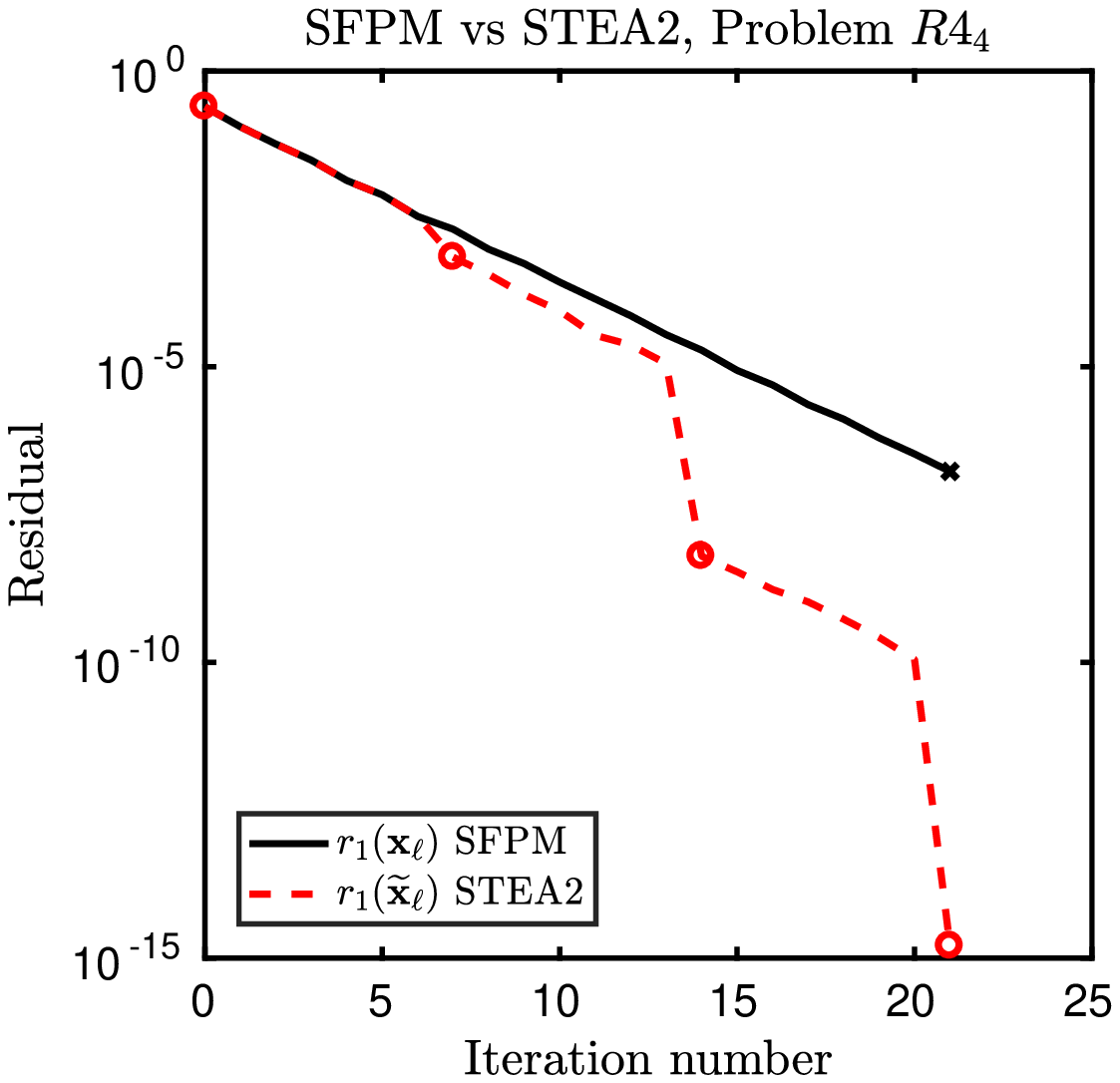} \hfill $\,$\\[1.5em]
	$\,$\hfill\includegraphics[width=0.4\columnwidth]{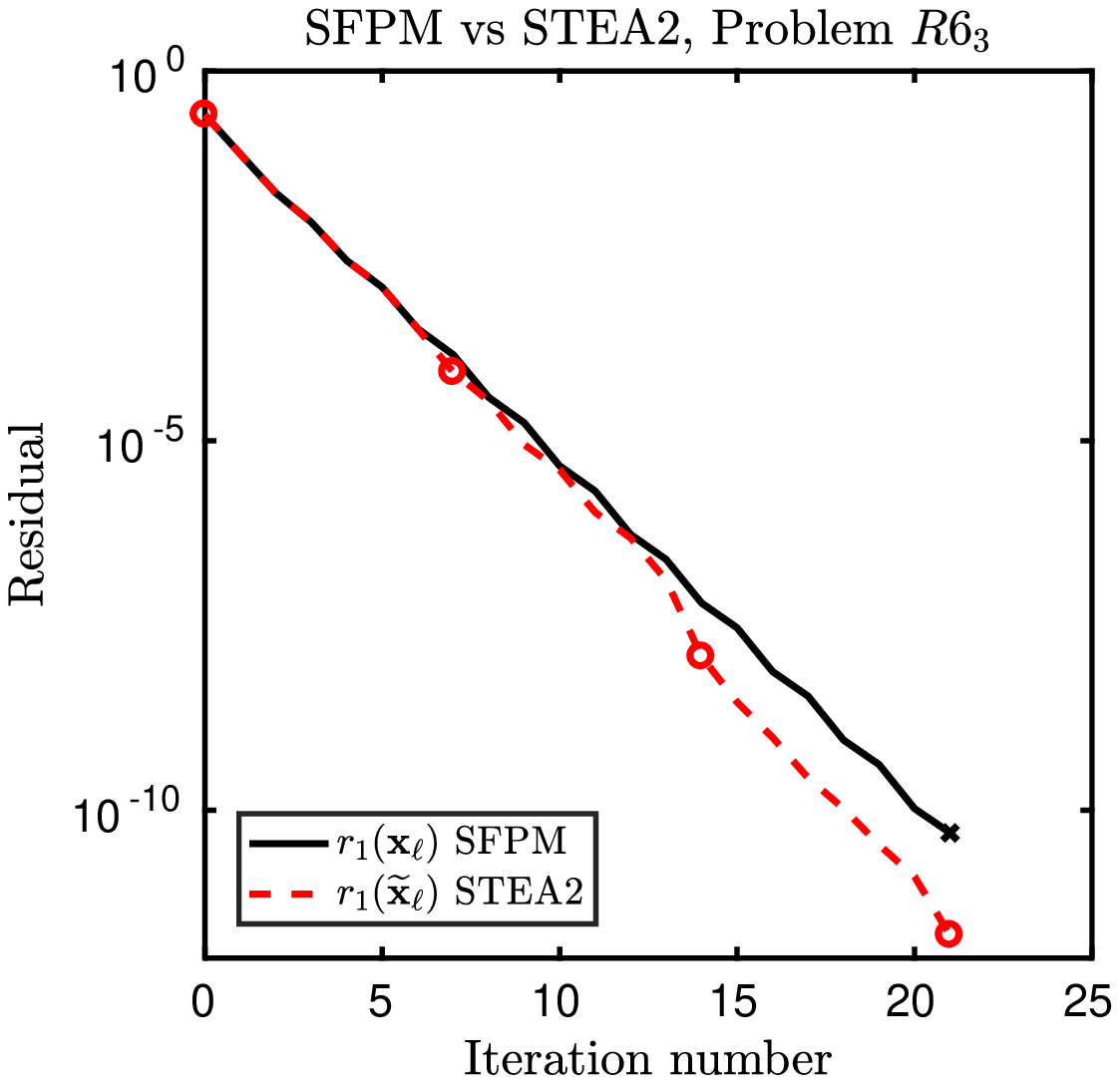}	\hfill
	\includegraphics[width=0.4\columnwidth]{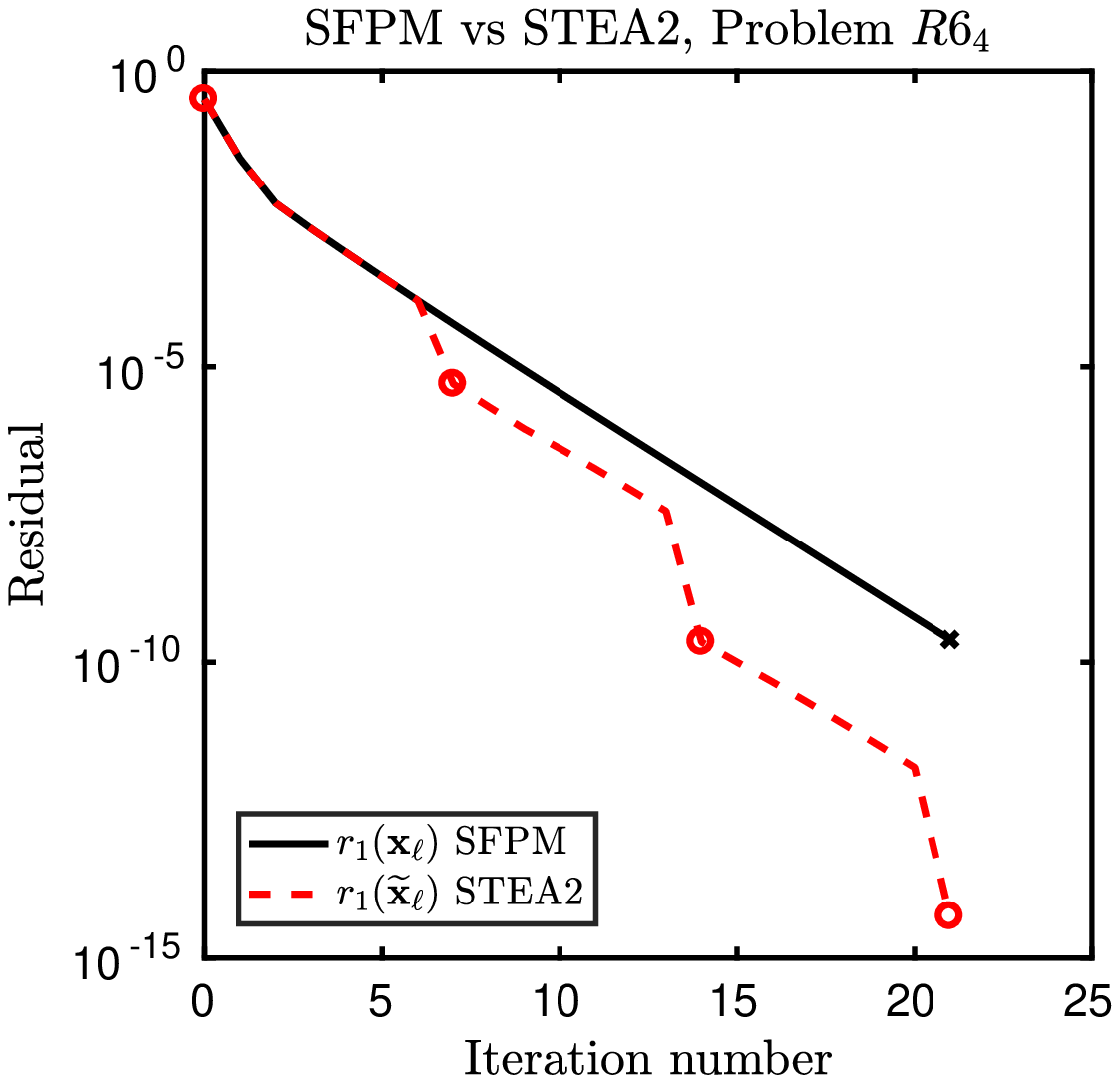}\hfill $\,$
	\caption{$\alpha=0.499\;\gamma=0$ Left column: worst acceleration performance, obtained on problems $R3_3,\;R4_{16},\;R6_3$ (top-bottom). Right column: best acceleration performance, obtained on problems $R3_1,\;R4_4,\;R6_4$ (top-bottom). }\label{fig:experiment1}
\end{figure}

\subsection{Problem Set  1} \label{sec:problemset1} In this section we use the benchmark set of 29 problems used in \cite{gleich2015,meini2017perron} that  consists of order-$3$  $n\times n \times n$ stochastic tensors where $n=3,4,6$.
\begin{figure}[htb!]
	\centering
	$\,$\hfill \includegraphics[width=0.4\columnwidth]{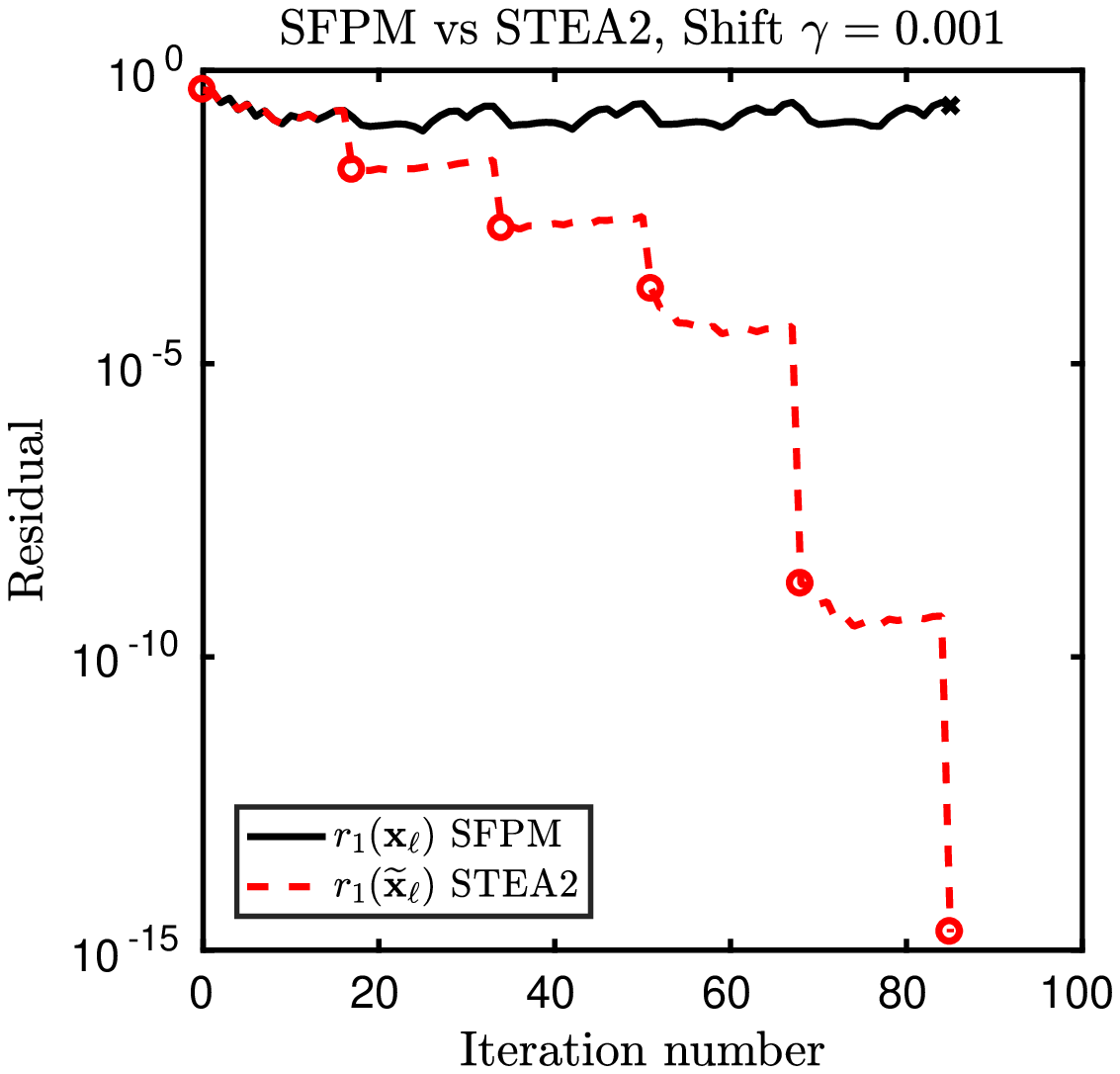} \hfill
	\includegraphics[width=0.4\columnwidth]{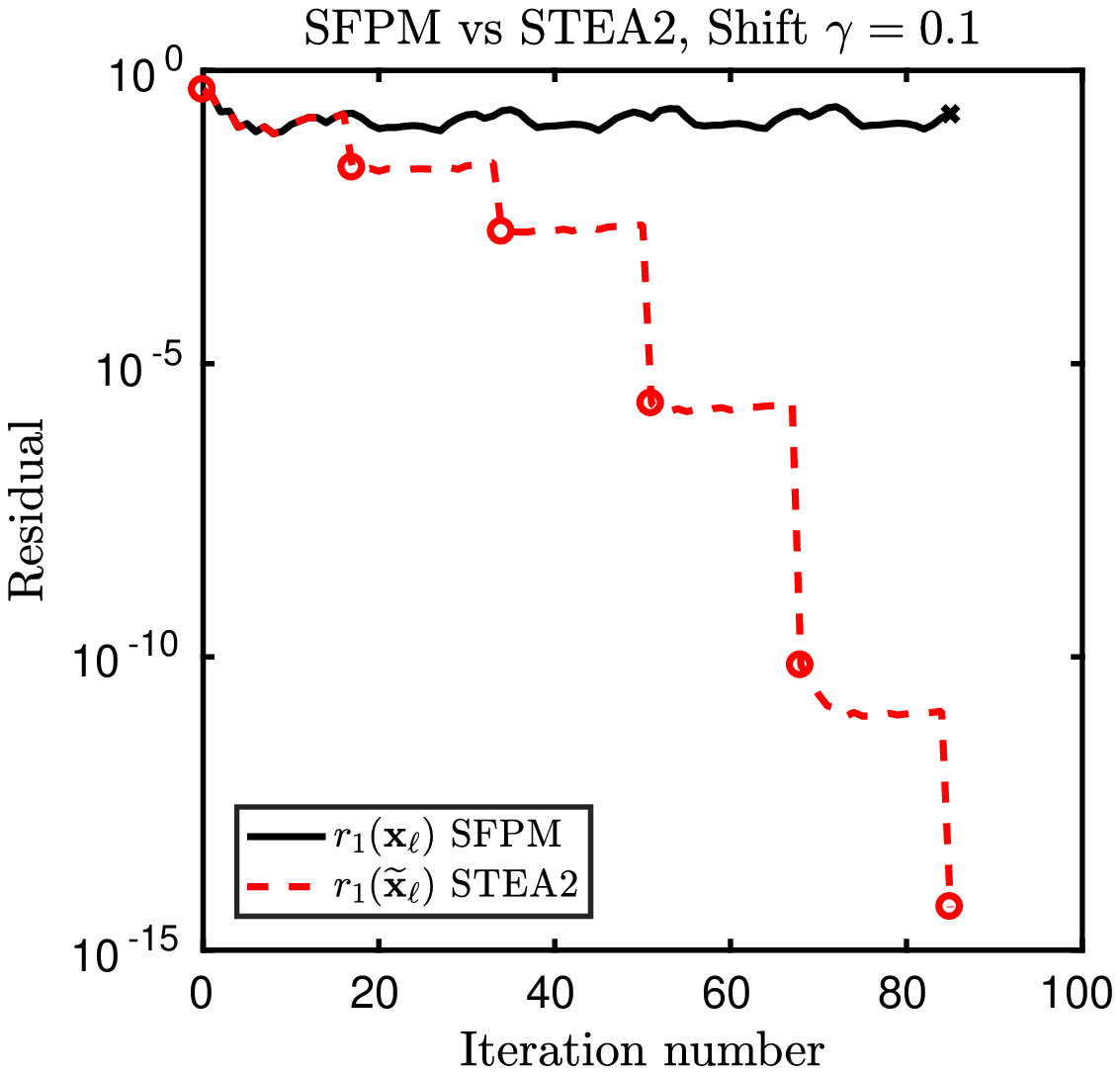}\hfill $\,$\\[1.5em]
	$\,$\hfill\includegraphics[width=0.4\columnwidth]{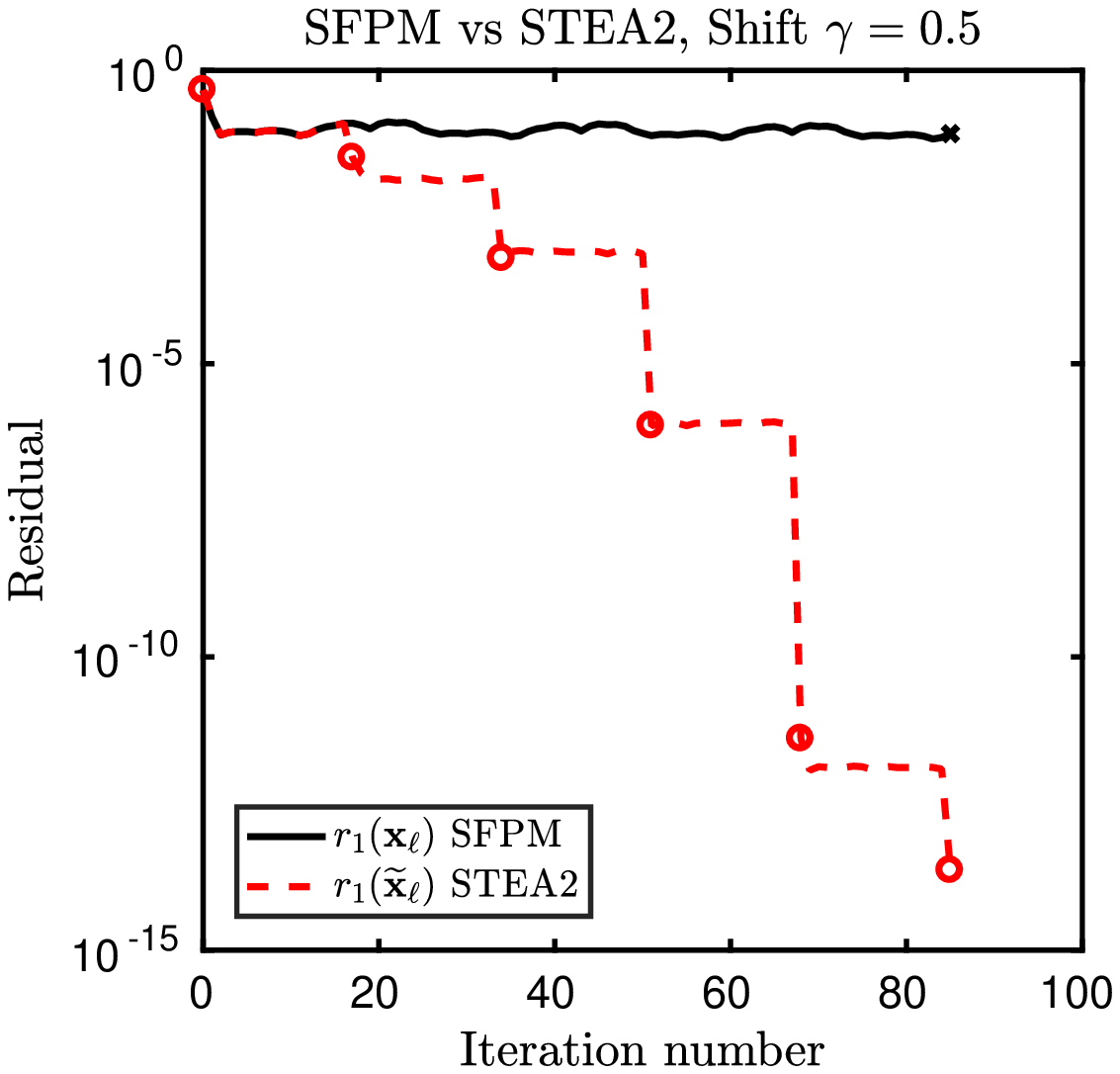}	\hfill
	\includegraphics[width=0.4\columnwidth]{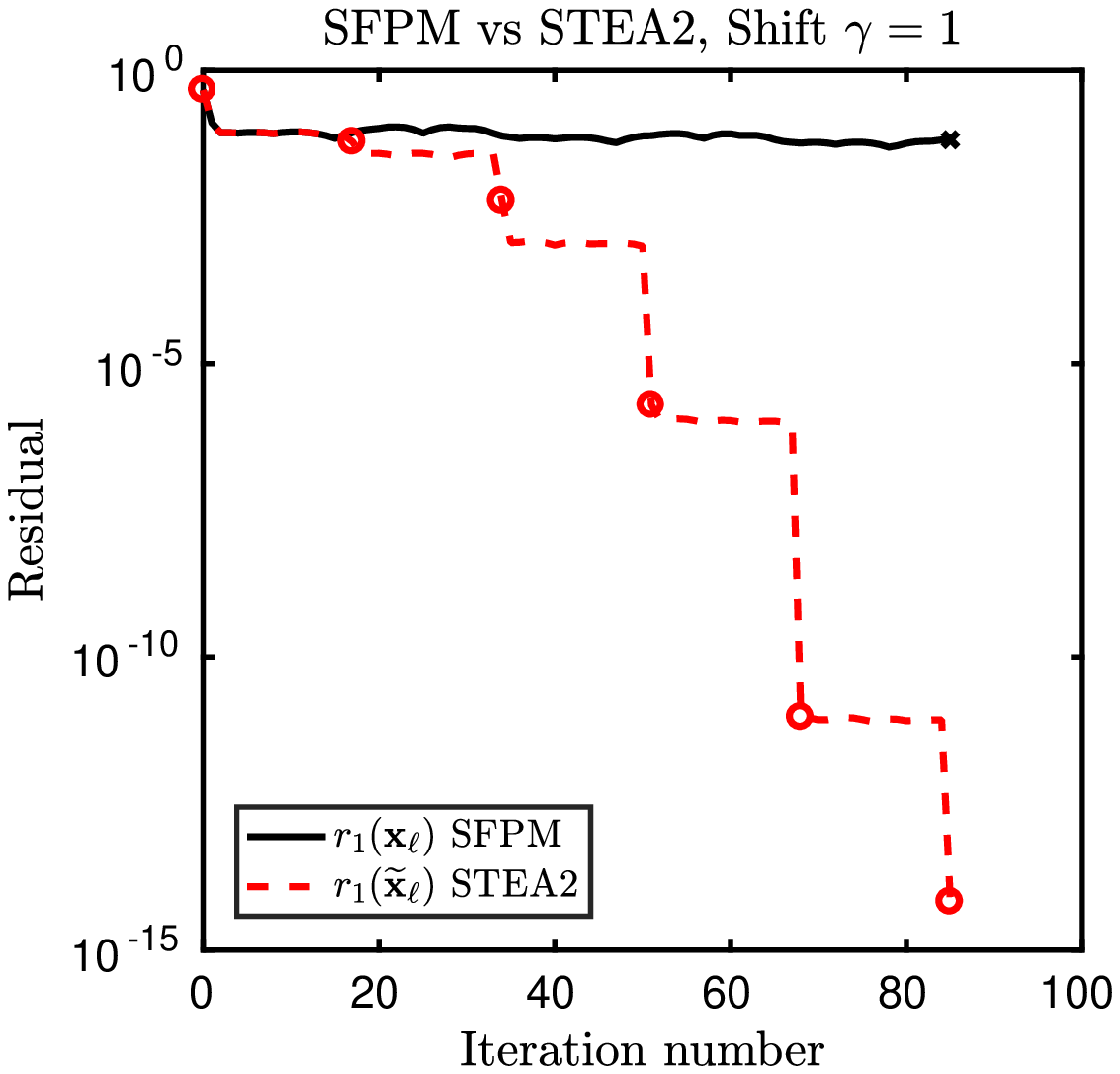}\hfill $\,$
	\caption{$\alpha=0.99$. Performance obtained on problem $R4_{19}$ for $\gamma \in \{0.001, 0.1, 0.5,1\}$.}\label{fig:experiment2}
\end{figure}

\begin{figure}[htb!]
	\centering
	\includegraphics[width=0.45\columnwidth]{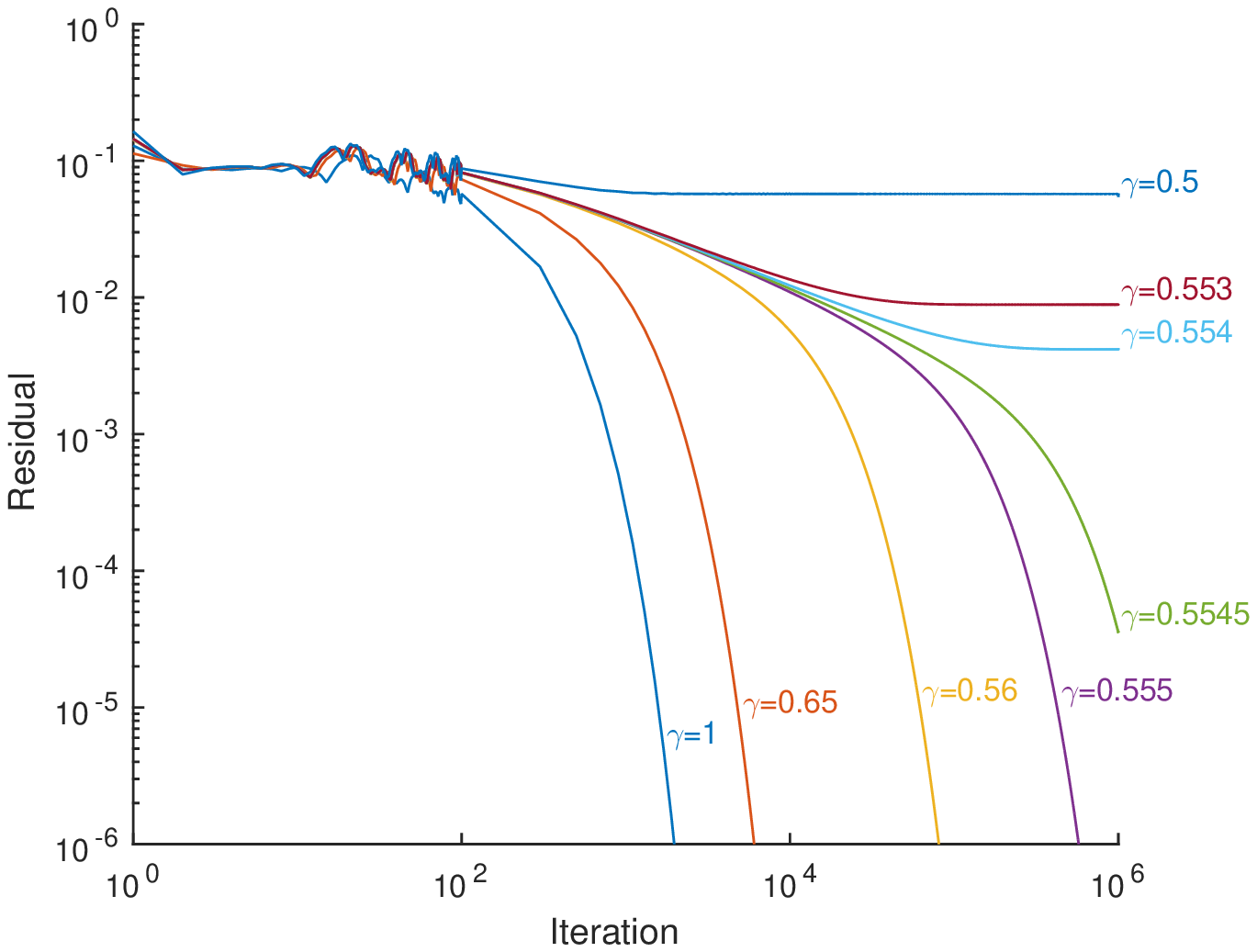}
	\vspace{-0.5cm}
	\caption{Results from \cite{gleich2015} on the test problem $R4_{19}$  for different values of $\gamma$.}\label{fig:plot_from_gleich}
\end{figure}

\subsubsection{Extrapolated shifted power method} \label{sec:extrapolated_SS_HOPM}
We tested Algorithm \ref{alg:restarted_method} on all the problems and, in Figure \ref{fig:experiment1}, we report for every $n=3,4,6$, the best and the worst  speed-up performance (in terms of computed residuals) for the shifted power method coupled with the restarted extrapolation method when $\alpha=0.499$ and $\gamma=0$. The choices of the parameters $2k$ and $\mathsf{cycles}$ in Algorithm \ref{alg:restarted_method} are, respectively, $2k=6$ and $\mathsf{cycles}=2$  or $\mathsf{cycles}=3$.

As it is clear from these results, using extrapolation techniques considerably improves the performance of the shifted power method for the multilinear {P}age{R}ank computation when $\alpha < 1/2$.
{Observe, moreover, that the usage of extrapolation techniques is useful even when $\alpha$ is out of the range of convergence as Figure \ref{fig:experiment2} shows ($2k=16$, $\mathsf{cycles}=5$). In that figure, we report detailed results for the test problem $R4_{19}$ from \cite{gleich2015}. This problem is particularly challenging to solve when $\gamma \leq 0.5$ and, indeed, it is shown in \cite{gleich2015} that for this range of the shifting parameter the power method fails to converge even after $10^6$ iterations. For completeness, we report in Figure \ref{fig:plot_from_gleich} the corresponding results from the original paper. As shown by Figure \ref{fig:experiment2}, the relevance of the extrapolation framework introduced is particularly clear in this case as, on the other hand, the extrapolated iterations converge even for very small values of $\gamma$. }

In order to have a more precise idea of the effectiveness of our approach, in Table \ref{table:solved_problems} we report the number of ``solved'' problems  when $\alpha=0.99$ and $\gamma=1$.
We say that a problem is ``solved'' if we are able to obtain a residual (\ref{eq:residual}) less than $10^{-8}$. Here the choices of $2k$ and $\mathsf{cycles}$ in Algorithm \ref{alg:restarted_method} range, respectively, between $6-20$ and $6-12$.

\begin{table}[hbt]
	\centering
	
	\begin{tabular}{|l|l|}
		\hline
		& Solved Problems \\ \hline
		$n=3$ & 4/5             \\ \hline
		$n=4$ & 15/19           \\ \hline
		$n=6$ & 3/5             \\ \hline
		Total & 22/29           \\ \hline
	\end{tabular}
	\caption{Performance for $\alpha=0.99$ and $\gamma=1$.\label{table:solved_problems}}
	\vspace{-0.5cm}
\end{table}
\subsubsection{Extrapolated inner-outer} \label{subsection:extrapolated_innerouter}
In the previous Section	\ref{sec:extrapolated_SS_HOPM} we showed  that the introduction of the simplified topological $\varepsilon$-algorithm in the shifted fixed-point method produces evident computational benefits for the solution of the multilinear PageRank problems in both the cases when $\alpha$ is inside or  outside  the range of convergence. The inner-outer method, at each iteration, solves a multilinear PageRank problem with $\alpha$ inside the convergence range (see \eqref{eq:in_out_it}); for this reason we strongly recommend the employment of extrapolation techniques in each inner step of the inner-outer method  when solved with the shifted fixed-point method.
Nevertheless, one could wonder whether the extrapolation strategy can speed-up not only the computations for the solution of each inner step of the inner-outer iteration, but also the convergence of the outer sequence. The results presented in this section address this issue.
\begin{figure}[p]
	\centering
	$\,$\hfill\includegraphics[width=0.4\columnwidth]{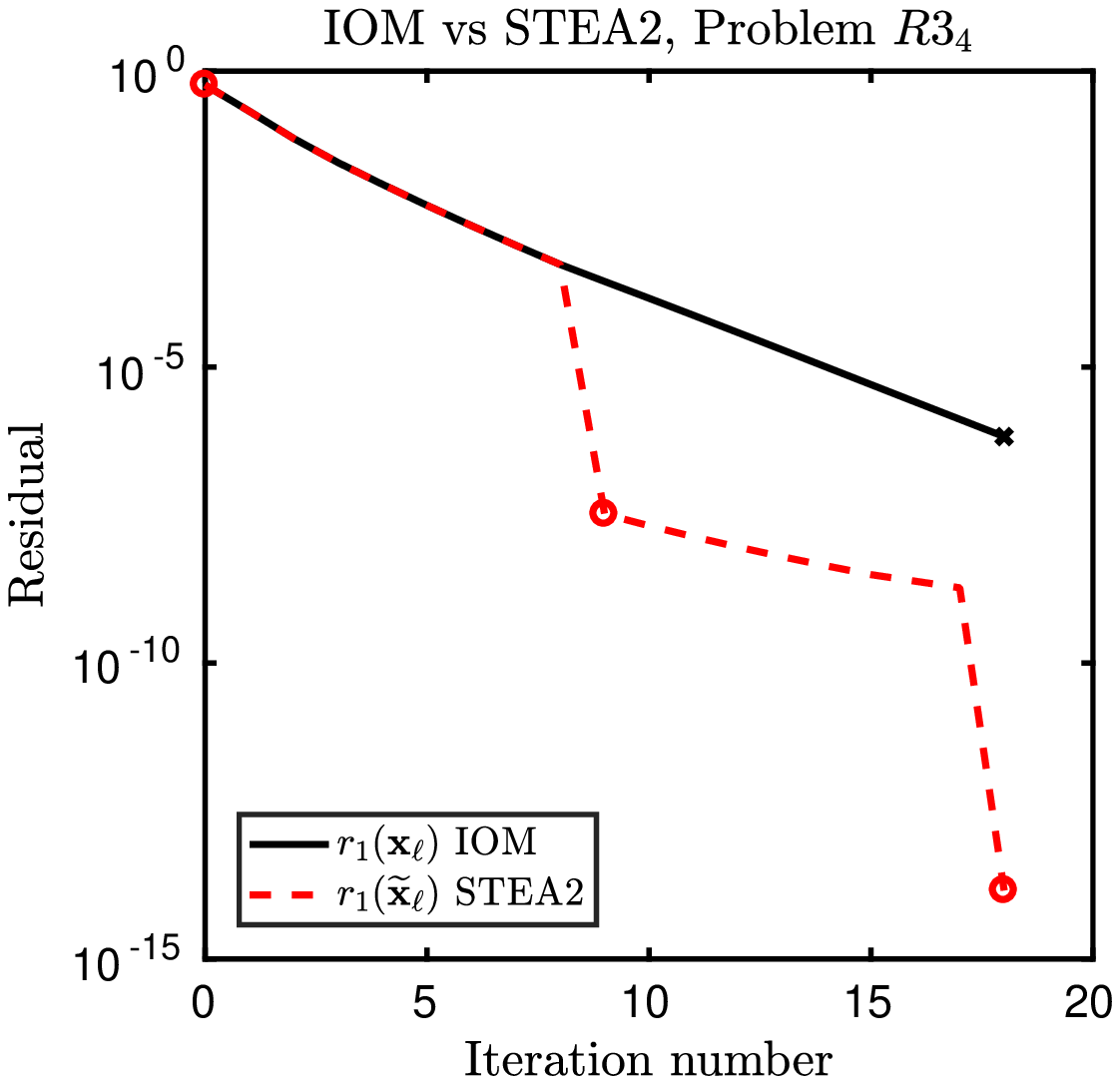} \hfill
	\includegraphics[width=0.4\columnwidth]{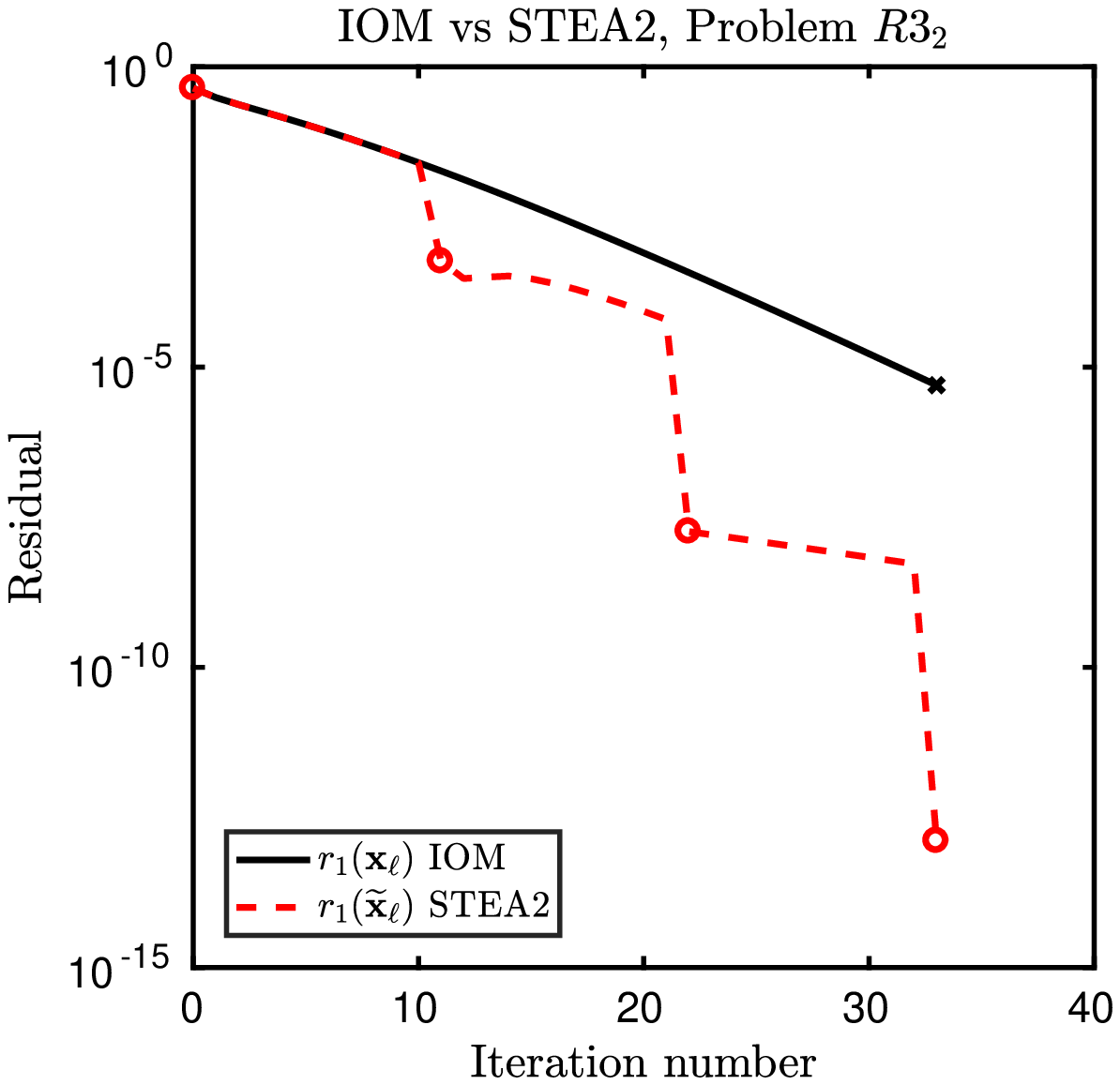} \hfill$\,$\\[1.5em]
	$\,$\hfill\includegraphics[width=0.4\columnwidth]{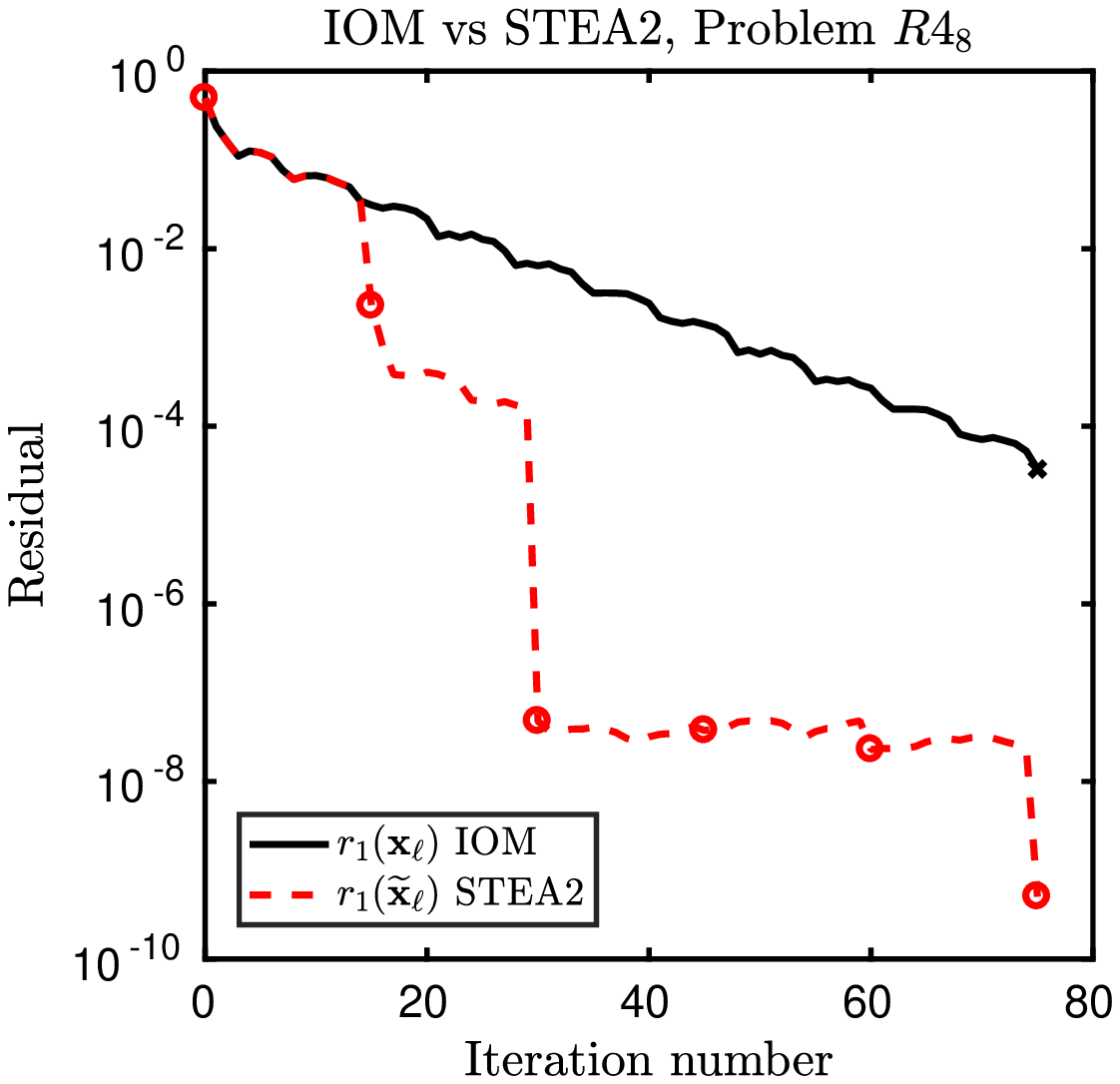}	 \hfill
	\includegraphics[width=0.4\columnwidth]{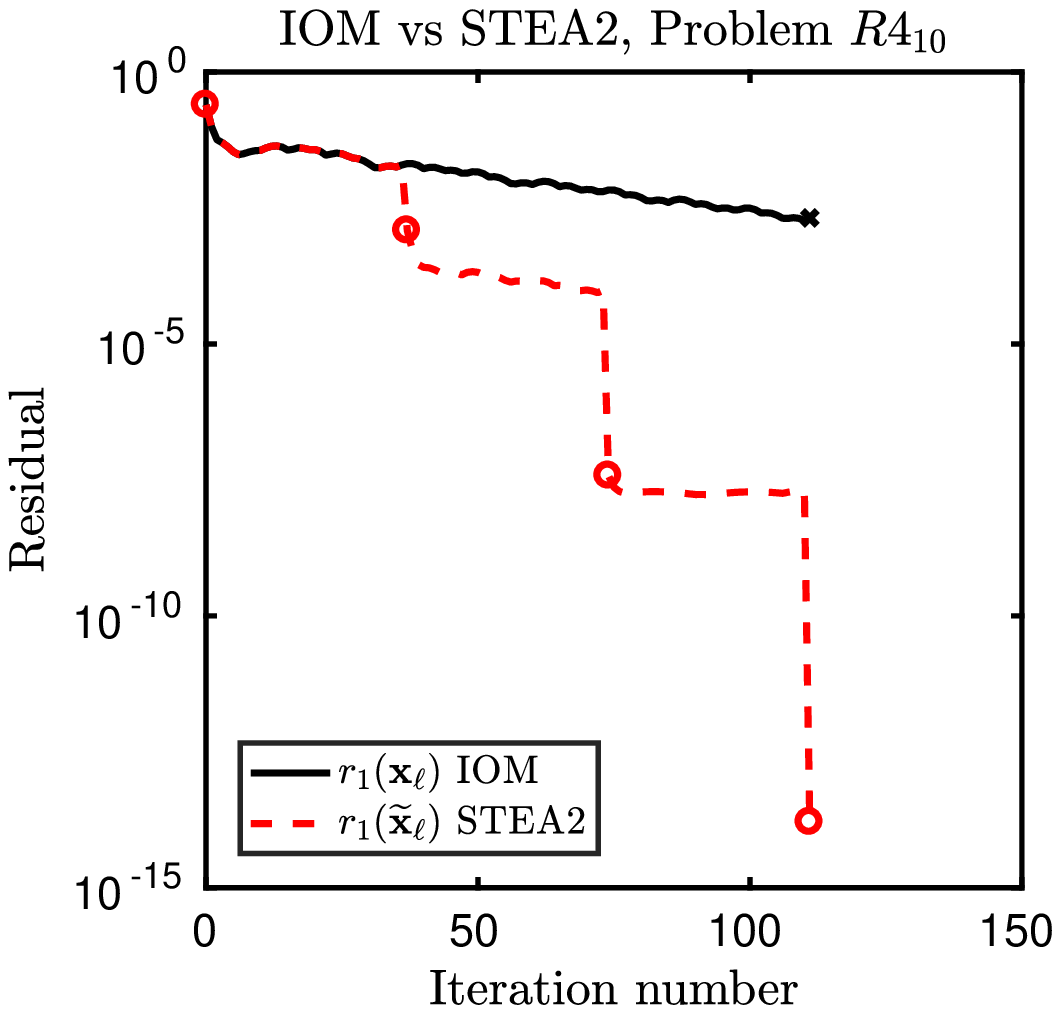}\hfill$\,$ \\[1.5em]
	$\,$\hfill\includegraphics[width=0.4\columnwidth]{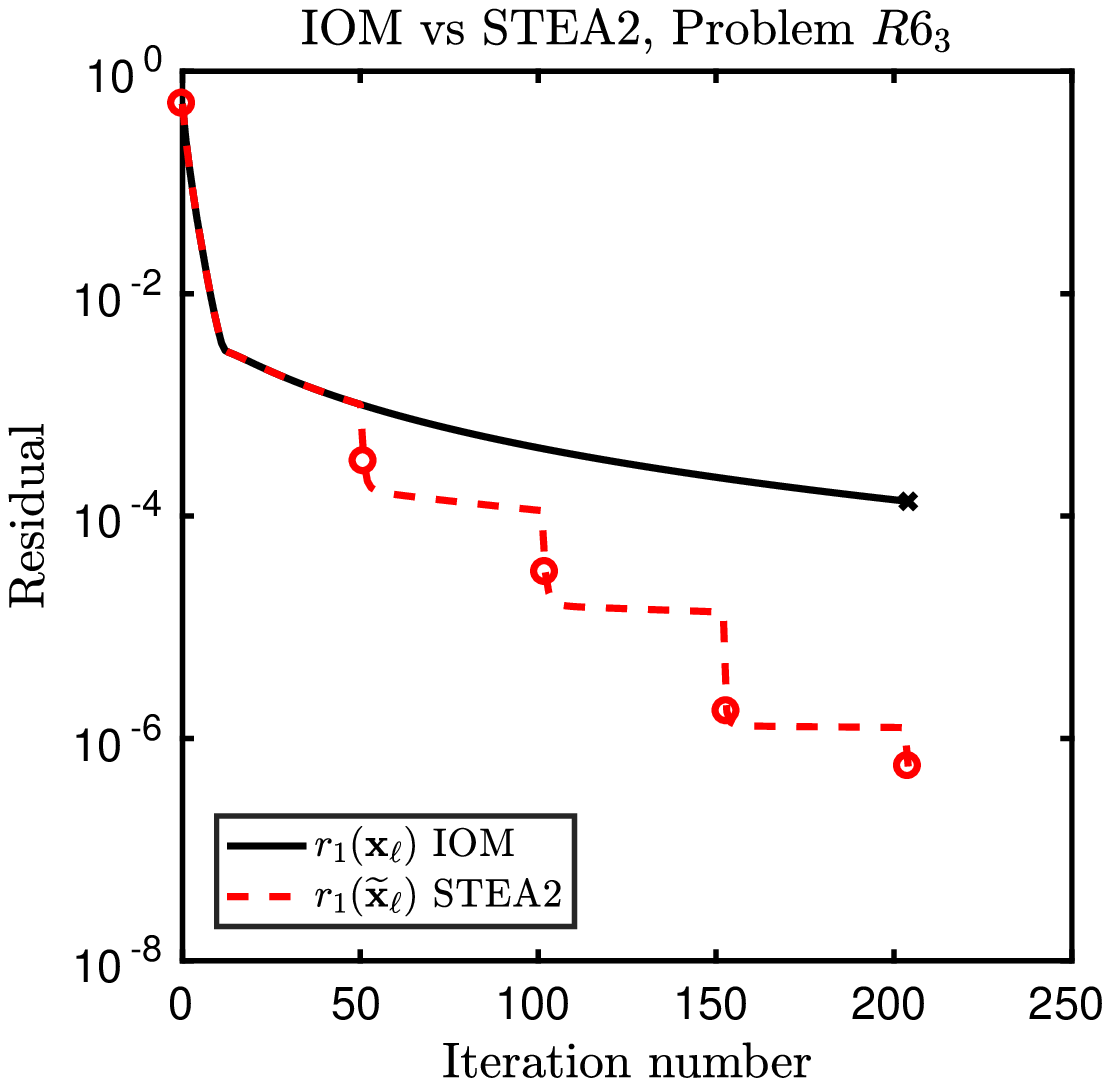}\hfill
	\includegraphics[width=0.4\columnwidth]{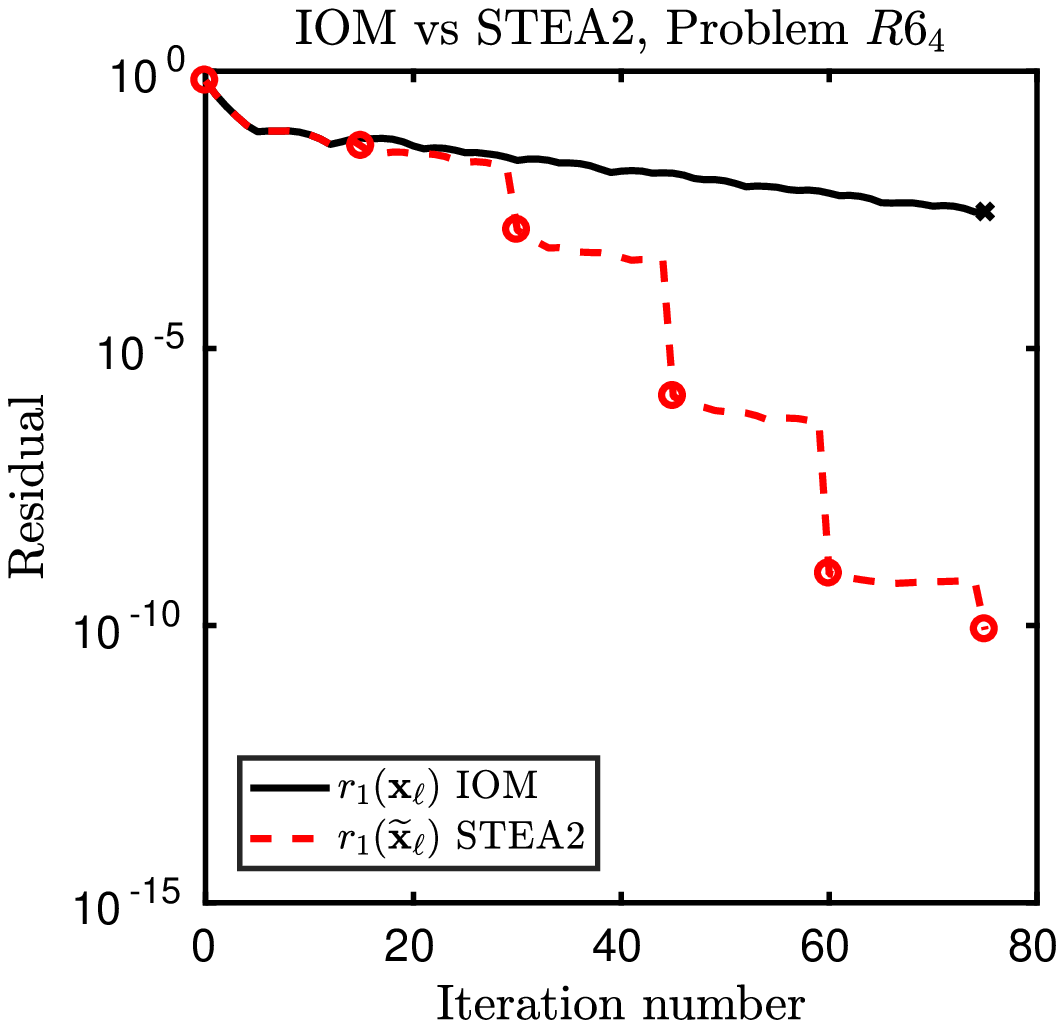} \hfill$\,$
	\caption{$\alpha=0.99.$ Left column: worst acceleration performance, obtained on problems $R3_4,\;R4_{8},\;R6_3$ (top-bottom). Right column: best acceleration performance, obtained on problems $R3_2,\;R4_{10},\;R6_4$ (top-bottom). $R6_3$ not solved. }\label{fig:experiment_inout}
\end{figure}

In Figure \ref{fig:experiment_inout}, we report best and worst performance of the inner-outer iteration \eqref{eq:in_out_it} with $\alpha=0.99$ when coupled with the restarted extrapolation method. We successfully solve  all the problems in our set   except for problem $R6_3$. The choices of the parameters $2k$ and $\mathsf{cycles}$ in Algorithm \ref{alg:restarted_method} range between $10-50$ and $3-9$, respectively.

\subsection{Problem Set 2} \label{sec:problemset2}
In order to  generate stochastic order-$m$ tensors of any desired {dimension} $n$ and produce statistics on a large number of datasets with different sizes, we consider the following procedure which uses the random graph generator CONTEST \cite{taylor2009contest}:

\begin{algorithm}
	\caption{Stochastic Tensor generator} \label{alg:random_graph_generator}
\begin{algorithmic}
	\State $n$ (size of the tensor), $m$ (modes of the tensor), \\ ${\bf g}=\{\mathtt{smallw}(n), \mathtt{gilbert}(n), \mathtt{erdrey}(n), \mathtt{pref}(n), \mathtt{geo}(n), \mathtt{lockandkey}(n), \mathtt{rank1}(n)\}$
	\For{$i=1,\dots, n^{m-2}$}
	\State Choose randomly an element ${\bf g}_r$ of ${\bf g}$
	\State	Set ${P}_{: \, , \,  n(i-1)+1:in }={\bf g}_r$
	\State	Transform  ${P}_{:\, ,\,  n(i-1)+1:in }$ into a stochastic matrix by rescaling the columns
	\EndFor
	\end{algorithmic}
\textbf{Output}: $\mathcal{P}$ is the $m$-th order stochastic tensor whose stochastic unfolding is $P$.
\end{algorithm}

\begin{figure}[h]
	\vspace{0.1cm}
	\centering
	{\includegraphics[width=0.7\linewidth]{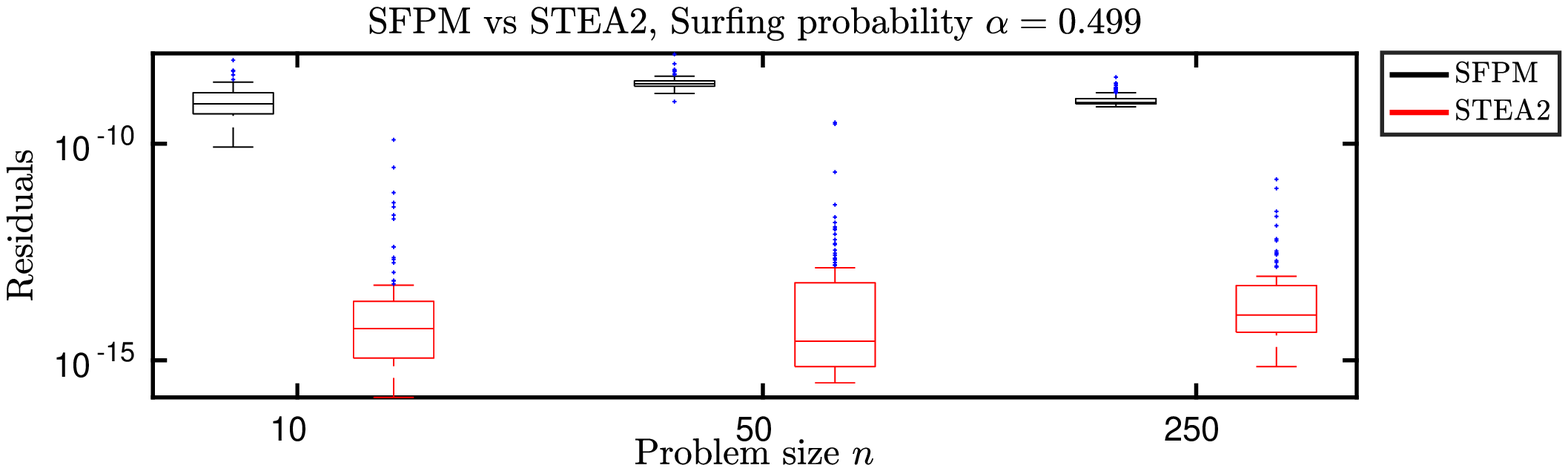}}	
	\caption{Median and quartiles of SFPM and STEA2 after $(2k+1) \times \mathsf{cycles}$ iterations on $100$ random tensors obtained by Algorithm \ref{alg:random_graph_generator}, with $\alpha=0.499$, $\gamma=0$, $2k=10,\;\mathsf{cycles}=2$, $n=10,50,250$ and $m=2$.}\label{fig:experiment1_rand_tens}
	\vspace{0.3cm}
	\centering
	{\includegraphics[width=0.7\columnwidth]{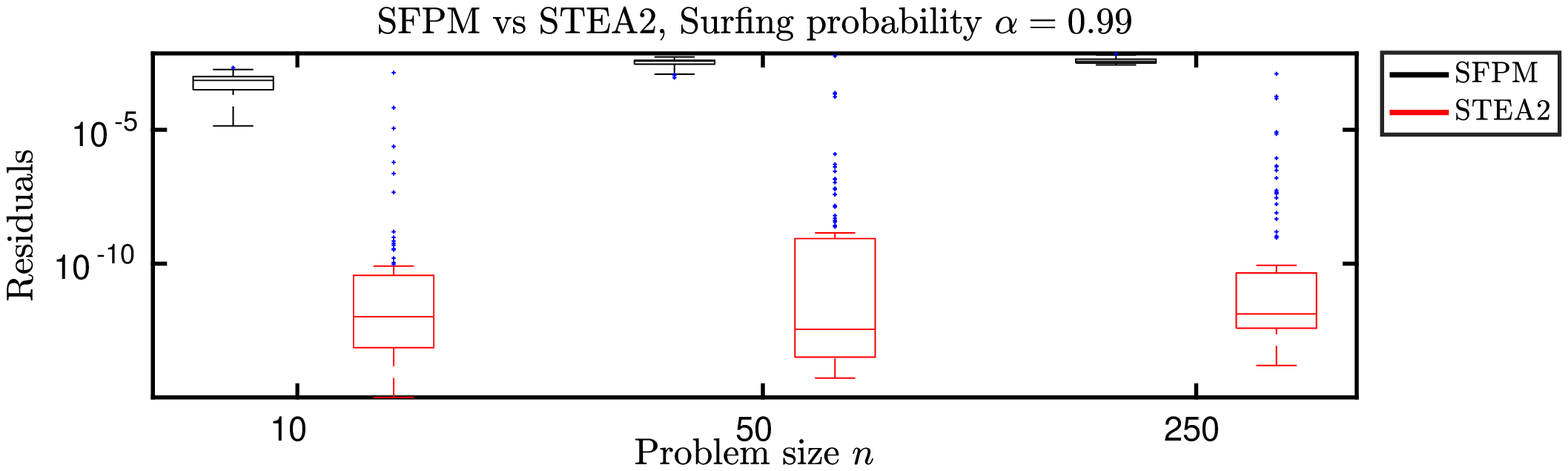}}	
	\caption{Median and quartiles of SFPM and STEA2 after $(2k+1) \times \mathsf{cycles}$ iterations on $100$ random tensors obtained by Algorithm \ref{alg:random_graph_generator}, with  $\alpha=0.99$, $\gamma=1$, $2k=28,\;\mathsf{cycles}=4$, $n=10,50,250$ and $m=2$.}\label{fig:experiment2_rand_tens}
\end{figure}

\begin{figure}[t]
	\vspace{0.3cm}
	\centering
	{\includegraphics[width=0.7\columnwidth]{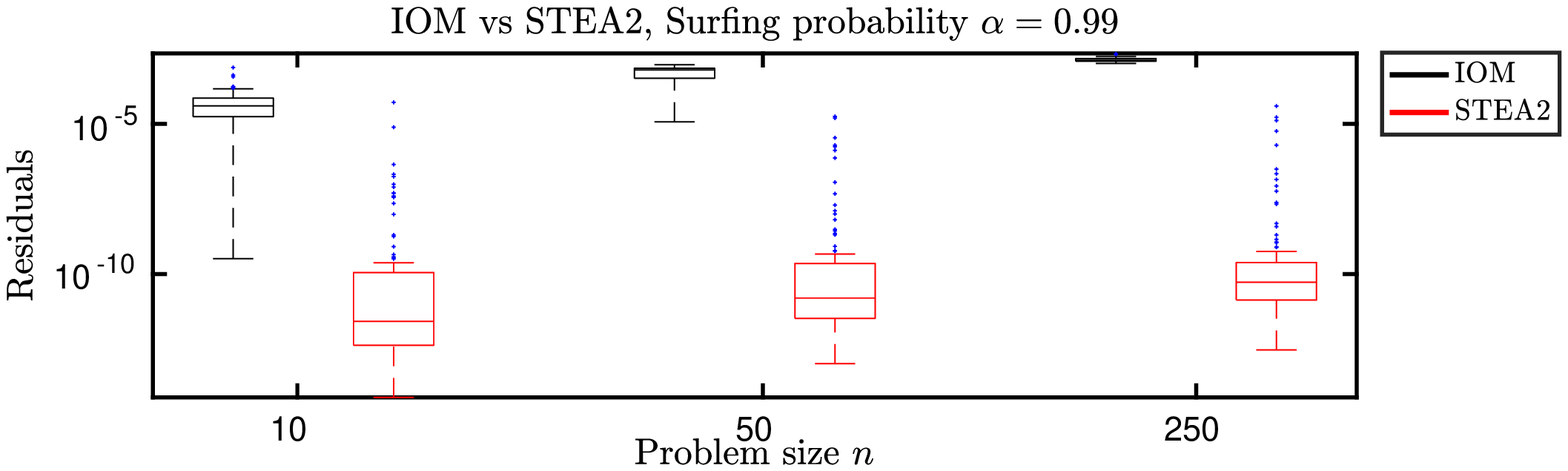}}	
	\caption{Median and quartiles of IOM and STEA2 after $(2k+1) \times \mathsf{cycles}$ iterations on $100$ random tensors obtained by Algorithm \ref{alg:random_graph_generator}, with $\alpha=0.99$, $2k=32,\;\mathsf{cycles}=4$, $n=10,50,250$ and $m=2$.}\label{fig:experiment3_rand_tens}
\end{figure}

{In the above procedure, the symbols  $$\mathtt{smallw}(n), \, \mathtt{gilbert}(n),\,  \mathtt{erdrey}(n),\,  \mathtt{pref}(n),\,  \mathtt{geo}(n),\,  \mathtt{lockandkey}(n),$$  
	denote random adjacency matrices  of dimension $n$ generated according to different random graph models described in the  CONTEST package, whereas  $\mathtt{rank1}(n)$ is the rank one matrix of all ones introduced in order to endow the tensor with a ``strong directionality'' feature which makes problems harder to be solved, as discussed in \cite{gleich2015}.
	In this way it is possible to provide statistics about the acceleration performance of the proposed techniques and investigate the dependence on the parameter $2k$ in Algorithm \ref{alg:restarted_method}  for increasing values of $n$ and $m$. As the obtained numerical results confirm, also in this case the use of extrapolation techniques significantly improves the performance of the shifted fixed-point method and of the inner-outer method. The  analysis carried out here highlights a very weak dependence  between  the number of vectors used to perform the extrapolation, i.e., the choice of $2k$ in Algorithm \ref{alg:restarted_method}, and the dimension of the problems. This feature demonstrates  that the proposed technique is particularly effective for solving problems of large scale. This is further supported by the analysis on large scale real-world  data carried out in the next Section \ref{sec:real_world_problems}.

\begin{figure}[t]
	\vspace{0.1cm}
	\centering
	{\includegraphics[width=0.7\linewidth]{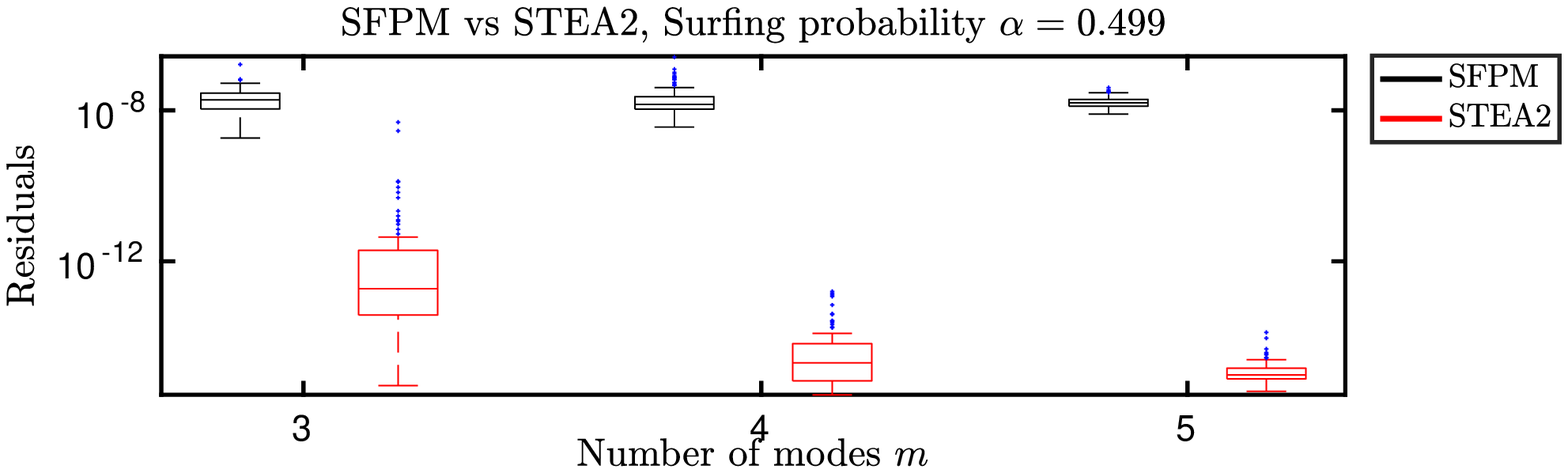}}	
	\caption{Median and quartiles of SFPM and STEA2 after $(2k+1) \times \mathsf{cycles}$ iteration on $100$ random tensors obtained by Algorithm \ref{alg:random_graph_generator}, with $\alpha=0.499$, $\gamma=0$, $2k=8,\;\mathsf{cycles}=2$, $n=10$ and $m=3,4,5$.}\label{fig:experiment4_rand_tens}
	\vspace{0.3cm}
	\centering
	{\includegraphics[width=0.7\columnwidth]{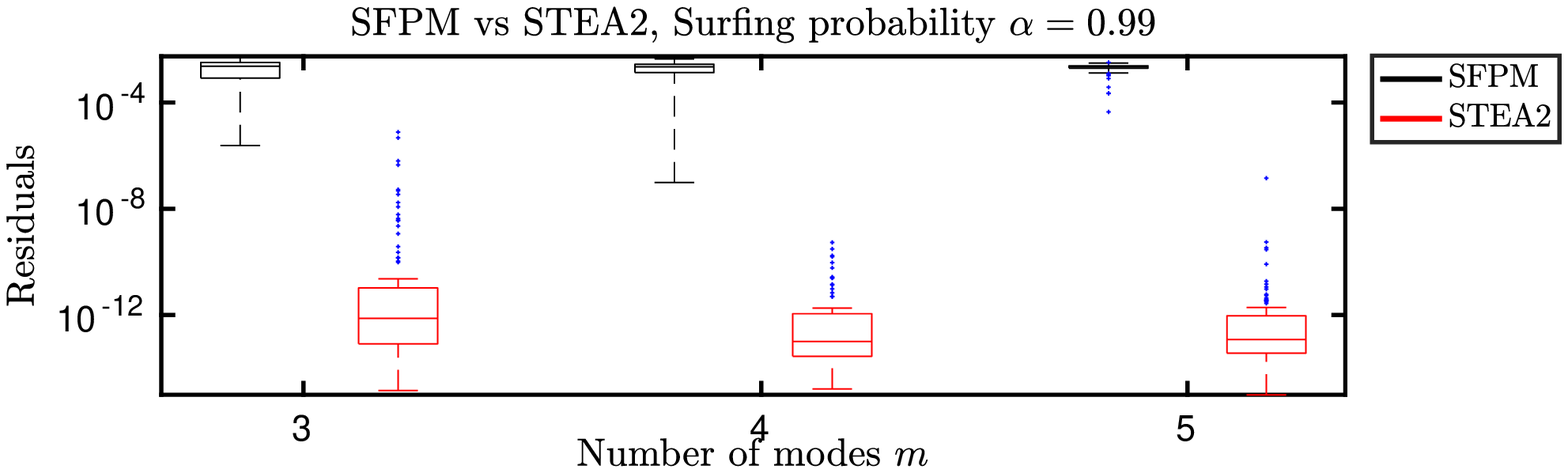}}	
	\caption{Median and quartiles of SFPM and STEA2 after $(2k+1) \times \mathsf{cycles}$ iteration on $100$ random tensors obtained by Algorithm \ref{alg:random_graph_generator}, with  $\alpha=0.99$, $\gamma=1$, $2k=22,\;\mathsf{cycles}=4$, $n=10$ and $m=3,4,5$.}\label{fig:experiment5_rand_tens}

\end{figure}	

	{In  Figure \ref{fig:experiment1_rand_tens} we show results, via Matlab's boxplots, on the acceleration performance of the shifted power method for the following choices of the parameters: $\alpha=0.499$, $\gamma=0$, $2k=10$ $\mathsf{cycles}=2$, $n=10,50,250$ and $m=2$. {In particular, we show  median and quartiles for the residual \eqref{eq:residual} after $(2k+1)\times \mathsf{cycles}$ iterations on 100 random example tests for both the shifted power method (in black) and the extrapolated sequence of Algorithm \ref{alg:restarted_method} (in red).  }
		In Figure \ref{fig:experiment2_rand_tens}, we show  analogous  results for the following choices of the parameters: $\alpha=0.99$, $\gamma=1$, $2k=28,\;\mathsf{cycles}=4$, $n=10,50,250$ and $m=2$.
		Whereas, in  Figure \ref{fig:experiment3_rand_tens}, we report the results of the acceleration performance with respect to the inner-outer method for the following choices of the parameters: $\alpha=0.99$, $2k=32,\;\mathsf{cycles}=4$, $n=10,50,250$ and $m=2$.

		Finally, Figures \ref{fig:experiment4_rand_tens} and \ref{fig:experiment5_rand_tens}, show how the acceleration performance of the shifted fixed-point method varies when the number of modes of the tensor increases. The choices of the parameters are the following: $n=10$, $\alpha=0.499$, $\gamma=0$, $2k=8$, $\mathsf{cycles}=2$ or $\alpha=0.99$, $\gamma=1$, $2k=22$, $\mathsf{cycles}=4$ and $m=3,4,5$.

		Figures \ref{fig:experiment1_rand_tens}-\ref{fig:experiment5_rand_tens} clearly show that the number of problems for which the proposed extrapolated algorithm provides major acceleration performance does not significantly change when $k$ and $\mathsf{cycles}$ in Algorithm \ref{alg:restarted_method} are fixed and either $n$ or $m$ increase. The independence of the parameter $k$ from the dimension $n$ and the number of modes $m$ is of utmost importance as this allows us to efficiently apply the proposed algorithm to large scale problems. This is also confirmed by real-world experiments of  Section \ref{sec:real_world_problems}.

	\subsection{Real-world datasets} \label{sec:real_world_problems}
	{In this section we report experimental results of the extrapolated  shifted fixed-point method when applied to a number of real world datasets  borrowed  from
		\cite{davis2011university,grindrod2016comparison}. Precisely, given an undirected  graph $G$ with $n$ nodes, we generate the third-order symmetric tensor $\mathcal C$ defined as follows

		\begin{equation*}
		\mathcal{C}_{i,j,k}= \begin{cases}
		1, \hbox{ if there is a three-cycle between nodes } i,j,k \\
		0, \hbox{ otherwise.}
		\end{cases}\, 
		\end{equation*}
		
		\begin{table}[hbt!]
		\centering  \fontsize{8}{5}\selectfont 
		\begin{tabular}{|lllll l lllll|}
			\hline
			Problem name & Source  &  Size & nnz(A) & nnz($\mathcal{C}$) & \hspace{3em} & Problem name & Source  &  Size & nnz(A) & nnz($\mathcal{C}$)\\
			\hline
			Bristol$^*$ & \cite{grindrod2016comparison} &  2892      & 9076 & 7722 & & minnesota& \cite{davis2011university} & 2642 &   6606   &  318    \\
			Cardiff$^*$& \cite{grindrod2016comparison}  &    2685   &  	8888 &  15186 & & NotreDame\_yeast& \cite{davis2011university} & 2114 &   4480   &  2075    \\
			Edinburgh$^*$ & \cite{grindrod2016comparison} &   1645 &   4292   &  1404   & & p2p-Gnutella04$^*$& \cite{davis2011university} & 10879   &    79988  &  5604    \\ 
			Glasgow$^*$ & \cite{grindrod2016comparison} &   1802   &  4568  &  1458 & & USpwerGrid& \cite{davis2011university} & 4941  &   13188     &   3906  \\
			Nottingham$^*$& \cite{grindrod2016comparison}  & 2066   &    6310 & 6162   & & wiki-Vote$^*$& \cite{davis2011university} &8297 &   201524     &  3650334  \\
			Erdos02 & \cite{davis2011university} &    6927    & 16944  & 14430 & & wing\_nodal& \cite{davis2011university} & 10937   &  150976      &  803082  \\
			EX6 & \cite{davis2011university} &    6545    & 295680  & 1774080 & & yeast& \cite{davis2011university} &2361 &     13828    &  35965  \\
			\hline
		\end{tabular}
		\caption{Datasets' information \label{table:synthetic_data}}
	\end{table}
	\begin{table}[t!]
		\centering \fontsize{10}{11}\selectfont 
		\begin{tabular}{|l|l|lll|lll|}
			\multicolumn{2}{l}{}                      & \multicolumn{3}{c}{\textbf{STEA2}} & \multicolumn{3}{c}{\textbf{SFPM}} \\ \hline
			& $\beta$ & Time(s)   & Iter       & Res  & Time(s)  & Iter      & Res \\
			\hline
			\multirow{3}{*}{Bristol}   & 0.1     &   \textbf{4.28} & \textbf{93} &   8.64e-10 &   \textbf{8.13} & \textbf{199} &   9.42e-09     \\
			& 0.3     &   1.70 & 38 &   2.80e-09 &   3.27 & 79 &   9.07e-09    \\
			& 0.6     &   1.21 & 27 &   1.09e-09 &   1.80 & 41 &   9.65e-09    \\ \hline
			\multirow{3}{*}{Cardiff}        & 0.1    &   \textbf{3.91} & \textbf{93} &   6.35e-09 &   \textbf{7.47} & \textbf{206} &   9.91e-09   \\
			& 0.3     &   1.48 & 38 &   5.83e-09 &   2.96 & 81 &   8.89e-09   \\
			& 0.6    &   1.06 & 27 &   3.45e-10 &   1.54 & 42 &   7.51e-09   \\ \hline
			\multirow{3}{*}{Edinburgh}      & 0.1      &   \textbf{1.67} & \textbf{93} &   2.84e-09 &   \textbf{2.75} & \textbf{205} &   9.62e-09
			\\
			& 0.3     &   0.40 & 38 &   2.48e-09 &   0.61 & 81 &   8.75e-09    \\
			& 0.6     &   0.42 & 27 &   1.56e-09 &   0.56 & 42 &   8.27e-09    \\ \hline
			\multirow{3}{*}{Glasgow}        & 0.1    &   3.23 & 155 &   6.64e-09 &   3.31 & 206 &   9.46e-09  \\
			& 0.3     &   0.72 & 57 &   5.45e-10 &   0.78 & 81 &   9.08e-09   \\
			& 0.6     &   0.46 & 27 &   7.58e-10 &   0.71 & 42 &   8.82e-09   \\ \hline
			\multirow{3}{*}{Nottingham}     & 0.1    &   4.11 & 155 &   3.86e-09 &   4.34 & 206 &   9.42e-09  \\
			& 0.3     &   0.98 & 38 &   1.93e-09 &   1.85 & 81 &   8.43e-09    \\
			& 0.6     &   0.69 & 27 &   1.17e-10 &   0.90 & 42 &   7.53e-09    \\ \hline
			\multirow{3}{*}{Erdos02}        & 0.1    &  29.61 & 124 &   8.02e-10 &  37.64 & 165 &   9.82e-09  \\
			& 0.3     &  13.82 & 57 &   1.31e-09 &  17.12 & 73 &   8.70e-09    \\
			& 0.6     &   6.54 & 27 &   4.50e-09 &   9.00 & 39 &   9.16e-09   \\ \hline
			\multirow{3}{*}{EX6}        & 0.1    &   \textbf{6.95 }& \textbf{31} &   2.40e-09 & \textbf{13.73} & \textbf{63} &   8.84e-09  \\
			& 0.3     &  10.31 & 38 &   2.90e-11 &  11.40 & 46 &   9.11e-09   \\
			& 0.6     &   6.28 & 27 &   1.21e-10 &   7.93 & 32 &   9.38e-09    \\ \hline
			\multirow{3}{*}{minnesota}      & 0.1     &   4.91 & 124 &   1.71e-09 &   7.56 & 221 &   9.47e-09    \\
			& 0.3     &   1.41 & 38 &   8.43e-09 &   2.87 & 84 &   8.51e-09   \\
			& 0.6     &   0.96 & 27 &   3.73e-10 &   1.48 & 43 &   7.37e-09   \\ \hline
			
			\multirow{3}{*}{ND\_yeast}      & 0.1     &   2.70 & 155 &   5.56e-09 &   2.95 & 229 &   9.52e-09   \\
			& 0.3     &   1.47 & 57 &   2.52e-11 &   1.82 & 85 &   9.71e-09   \\
			& 0.6     &   0.45 & 27 &   4.21e-09 &   0.55 & 43 &   8.51e-09  \\ \hline
			
			\multirow{3}{*}{p2p-Gnutella04} & 0.1     &  \textbf{35.91} & \textbf{62} &   5.18e-09 &  \textbf{71.63} & \textbf{127} &   9.96e-09   \\
			& 0.3     &  22.44 & 38 &   2.32e-09 &  31.28 & 55 &   9.05e-09    \\
			& 0.6     &  15.82 & 27 &   7.84e-10 &  20.02 & 35 &   7.75e-09   \\ \hline
			\multirow{3}{*}{USpowerGrid} & 0.1     &  19.36 & 155 &   3.20e-09 &  25.94 & 210 &   9.65e-09   \\
			& 0.3     &   4.78 & 38 &   1.96e-09 &   9.73 & 82 &   8.67e-09    \\
			& 0.6    &   3.34 & 27 &   5.58e-10 &   5.01 & 42 &   9.11e-09    \\ \hline
			
			\multirow{3}{*}{wiki-Vote}      & 0.1     &   \textbf{32.32} & \textbf{93} &   1.39e-10 &  \textbf{57.22} & \textbf{170} &   9.72e-09   \\
			& 0.3     &  13.10 & 38 &   1.37e-09 &  23.41 & 65 &   8.61e-09   \\
			& 0.6     &   9.22 & 27 &   9.83e-10 &  11.95 & 36 &   8.85e-09    \\ \hline

			\multirow{3}{*}{wing\_nodal}    & 0.1     &  \textbf{54.76} & \textbf{93} &   6.15e-10 & \textbf{106.82} & \textbf{187} &   9.43e-09  \\
			& 0.3     &  22.38 & 38 &   7.94e-09 &  42.43 & 74 &   9.23e-09   \\
			& 0.6     &  12.09 & 18 &   3.22e-09 &  22.09 & 39 &   8.06e-09    \\ \hline
			\multirow{3}{*}{yeast}          & 0.1    &   \textbf{3.06} & \textbf{93} &   1.20e-09 &   \textbf{5.81} & \textbf{212} &   9.97e-09   \\
			& 0.3     &   0.69 & 38 &   1.59e-09 &   1.31 & 79 &   9.22e-09   \\
			& 0.6      &   0.80 & 27 &   4.53e-09 &   1.14 & 41 &   7.24e-09    \\ \hline
		\end{tabular}
		\caption{Performance on the real-world datasets of Table \ref{table:synthetic_data}\label{table:real_world_results}}
	\end{table}
{Following the construction proposed in  \cite{benson2015tensor,gleich2015}, we transform $\mathcal C$ into a new stochastic tensor $\mathcal R$ as follows.
			Let $S$ be   the mode-one unfolding of $\mathcal{C}$ that has been normalized to be a sub-stochastic matrix by scaling its columns, let $A$ be the adjacency matrix of $G$ and let $D$ be the diagonal matrix of the degrees of $G$, $D_{ii} = \sum_j A_{ij}$. Finally, let $M=A^TD^{+}$, where $D^{+}$ is the Moore-Penrose pseudoinverse of $D$. We define $\mathcal R$ as the order-3 tensor whose mode-1 unfolding is 
			\begin{equation}
			R=\beta[S+\mathbf{v}\, \mathrm{dangling}(S)]+(1-\beta)[M+\mathbf{v}\, \mathrm{dangling}(M)]\otimes\mathbf{e}\, ,
			\end{equation}
			where $0\leq\beta\leq1$,  $\mathbf{v}\in \Omega_+$ is a positive stochastic vector and where, for a  matrix $B \in \mathbb{R}^{n \times n}$ and the all-ones vector $\mathbf{e}$,  we let  $\mathrm{dangling}(B):=\mathbf{e}^T-\mathbf{e}^TB$.}

		We report in Table \ref{table:synthetic_data} the details of the graphs we took into account; when the graph $G$ is directed -- marked with an asterisk  -- we considered its undirected version: if $A$ is the adjacency matrix of $G$,  we consider the graph $G'$ corresponding to the adjacency matrix $A'=\mathrm{sign}(A+A^T)$, whose entries are $(A')_{ij} = \mathrm{sign}(a_{ij}+a_{ji})$.	
		In Table \ref{table:real_world_results} we report the obtained numerical results when $\alpha=0.99$, $\gamma=1$,  $2k=8$ if $\beta=0.6$, $2k=18$ if $\beta=0.3$ and $2k=30$ if $\beta=0.1$.

		{Let us stress that the execution times we report here include the running time of the extrapolation routines.
		It is possible to clearly appreciate the computational benefits resulting from the introduction of the extrapolation techniques: the computational times and the number of iterations needed to produce a residual not larger than $10^{-8}$ are always reduced; for some  problems, for which the shifted fixed-point method is particularly slow ($\beta=0.1$), the iterations and the execution times reduce dramatically, up to the 230\%, obtained for the dataset ``yeast''. We highlighted in bold the most relevant speed-up performance in Table \ref{table:real_world_results}.}

\section{Conclusions}
		We showed that the use of extrapolation techniques for the computation of the multilinear {P}age{R}ank using fixed-point iterations improves substantially the efficiency and effectiveness on the test problems we considered at a limited additional cost.  In particular, we have shown that extrapolated versions of the shifted power and inner-outer methods are able to solve certain pathological test problems where the original methods consistently fail. Moreover, the performance of the methods (in terms of execution time) remarkably improves, allowing to address very large problems.  Finally, let us point out that, as it is formulated, the multilinear {P}age{R}ank problem is equivalent to the computation of a $Z$-eigenvector of a stochastic tensor for which the shifted power  method introduced in \cite{kolda2011shifted} can be used. In that work the authors observed that the shifted power method could suffer of an extremely low rate of convergence affecting  its employment in large scale applications. For this reason, they  have raised the issue regarding the possibility of suitably speeding-up the shifted power method. The  numerical results obtained and presented in this work allow us to answer positively to their question when the shifted power method is used for the computation of the multilinear {P}age{R}ank solution and is coupled with the Simplified Topological $\varepsilon$-algorithm,  encouraging the numerical community to use the proposed techniques when dealing with problems similar to those presented here. 
		We intend to come back to the theoretical analysis concerning the justification of the remarkable acceleration performance obtained here in a forthcoming work.
	
\section*{Acknowledgments}		
S.C. and M.R.-Z. are members of the INdAM Research group GNCS, which partially supported this work. \\The work of F.T. was funded by the European Union's Horizon 2020 research and innovation programme under the the Marie Sk\l odowska-Curie Individual Fellowship ``MAGNET'' no.\ 744014.

\bibliographystyle{abbrv}
\bibliography{Mult_pr}
\end{document}